\documentclass[11pt,reqno]{amsart}
\usepackage[utf8]{inputenc}
\usepackage{amsmath,amsthm,amssymb,amsopn}
\usepackage[pagebackref,colorlinks,linkcolor=red,citecolor=blue,urlcolor=blue,hypertexnames=false]{hyperref}
\usepackage{amsrefs}
\usepackage{amscd}

\input{xy}
\xyoption{all}
\newdir{ >}{{}*!/-5pt/@{>}}
\usepackage{tikz}
\usetikzlibrary{matrix,arrows}
\usepackage{txfonts}

\newcommand\restr[2]{{
  \left.\kern-\nulldelimiterspace 
  #1 
  \vphantom{\big|} 
  \right|_{#2} 
  }} 

\newcommand{\comment}[1]{}

\def\N{\mathbb{N}} 
\def\Z{\mathbb{Z}} 
\def\Q{\mathbb{Q}}

\theoremstyle{plain}
\newtheorem{teo}[equation]{Theorem} 
\newtheorem{thm}[equation]{Theorem}
\newtheorem{lema}[equation]{Lemma}
\newtheorem{lem}[equation]{Lemma}
\newtheorem{coro}[equation]{Corollary} 
\newtheorem{prop}[equation]{Proposition}

\theoremstyle{definition}

\newtheorem{defi}[equation]{Definition} 

\newtheorem{ex}[equation]{Example}
\newtheorem{stan}[equation]{Standing assumptions}
\newtheorem{conve}[equation]{Convention}
\theoremstyle{remark} 

 \newtheorem{rem}[equation]{Remark}
 \newtheorem{nota}[equation]{Notation}
  \numberwithin{equation}{section}

\newcommand{\cB}{\mathcal B}

\newcommand{\cE}{\mathcal E}
\newcommand{\cF}{\mathcal F}

\newcommand{\cK}{\mathcal K}

\newcommand{\cP}{\mathcal P}

\newcommand{\cR}{\mathcal R}
\newcommand{\cT}{\mathcal T}
\newcommand{\cU}{\mathcal U}

\def\fA{\mathfrak{A}}

\newcommand{\aha}{{{\rm Alg}_\ell}}
\newcommand{\lra}{\longrightarrow}
\newcommand{\iso}{\overset{\sim}{\lra}}
\newcommand{\simh}{\approx}
\newcommand{\onto}{\twoheadrightarrow}
\def\reg{\operatorname{reg}}
\def\sing{\operatorname{sing}}
\def\sink{\operatorname{sink}}
\def\inf{\operatorname{inf}}
\def\sour{\operatorname{sour}}
\def\triqui{\vartriangleleft}
\def\mspan{\operatorname{span}}
\def\supp{\operatorname{supp}}
\def\inc{\operatorname{inc}}
\def\can{\operatorname{can}}
\def\rk{\operatorname{rk}}
\def\tors{\operatorname{tors}}
\def\diag{\operatorname{diag}}
\def\Gl{\operatorname{GL}}
\def\GL{\Gl}
\def\ad{\operatorname{ad}}

\def\ev{\operatorname{ev}}
\def\id{\operatorname{id}}
\renewcommand{\path}{\cP}

\newcommand{\coker}{{\rm Coker}}
\renewcommand{\ker}{{\rm Ker}}

\newcommand{\ju}{\mathfrak{j}}

\def\Ext{\operatorname{Ext}}
\def\cExt{\mathcal{E}xt}
\def\ext{\Ext}
\def\tor{\operatorname{Tor}}
\def\Hom{\operatorname{Hom}}

\title{Algebraic bivariant $K$-theory and Leavitt path algebras}
\date{}

\author{Guillermo Corti\~nas}
\author{Diego Montero}
\email{gcorti@dm.uba.ar, dmontero@dm.uba.ar}
\urladdr{http://mate.dm.uba.ar/\~{}gcorti}
\address{Dep. Matem\'atica-IMAS, FCEyN-UBA\\ Ciudad Universitaria Pab 1\\
C1428EGA Buenos Aires\\ Argentina}

\thanks{Both authors were supported by CONICET and by grants UBACyT 20021030100481BA and  PICT 2013-0454. Corti\~nas research was supported by grant MTM2015-65764-C3-1-P (Feder funds).}

\begin{document}
\begin{abstract}
We investigate to what extent homotopy invariant, excisive and matrix stable homology theories help one distinguish between the Leavitt path algebras $L(E)$ and $L(F)$ of graphs $E$ and $F$ over a commutative ground ring $\ell$. We approach this by studying the structure of such algebras under bivariant algebraic $K$-theory $kk$, which is the universal homology theory with the properties above. We show that under very mild assumptions on $\ell$, for a graph $E$ with finitely many vertices and reduced incidence matrix $A_E$, the structure of $L(E)$ in $kk$ depends only on the groups $\coker(I-A_E)$ and $\coker(I-A_E^t)$. We also prove that for Leavitt path algebras, $kk$ has several properties similar to those that Kasparov's bivariant $K$-theory has for $C^*$-graph algebras, including analogues of the Universal coefficient and K\"unneth theorems of Rosenberg and Schochet.
\end{abstract}

\maketitle

\section{Introduction}\label{sec:intro}
This article is the first part of a two part project motivated by the classification problem for Leavitt path algebras \cite{alps}.  We consider homological invariants of such algebras; we investigate to what extent they help one distinguish between them. In this first part we investigate general graphs and their algebras over general commutative ground rings; the second part \cite{dwclass} focuses mostly on purely infinite simple unital algebras over a field. We fix a commutative ring $\ell$ and write $L(E)$ for the Leavitt path algebra of a graph $E$ over $\ell$. Here a \emph{homology theory} of the category $\aha$ of algebras is simply a functor $X:\aha\to\cT$ with values in some triangulated category $\cT$. If $S$ is a set and $A\in \aha$, we write $M_SA$ for the algebra of those matrices
$M:S\times S\to A$ which are finitely supported. We call a homology theory $X$ \emph{$S$-stable} if for $s\in S$ and $A\in\aha$, the inclusion $\iota_s:A\to M_SA$, $\iota_s(a)=\epsilon_{s,s}\otimes a$ induces an isomorphism $X(\iota_s)$. Write $E^0$ and $E^1$ for the sets of vertices and edges of the graph $E$. We call $X$ $E$-stable if it is $E^0\sqcup E^1\sqcup \N$-stable. Thus if $E^0$ and $E^1$ are both countable, $E$-stability is the same as stability with respect to $M_\infty=M_\N$. We are interested in those homology theories which are excisive, (polynomially) homotopy invariant and $E$-stable. For example Weibel's homotopy algebraic $K$-theory $KH$ has all these properties and, if $\ell$ is either $\Z$ or a field, then  $KH_*(L(E))=K_*(L(E))$ is Quillen's $K$-theory. There is also a universal homology theory with all the above properties, $j:\aha\to kk$ (\cite{kkwt},\cite{emathesis}); this is the bivariant $K$-theory of the title. For two algebras $A,B\in\aha$, the statement  $j(A)\cong j(B)$ is equivalent the statement that $X(A)\cong X(B)$ for any excisive, homotopy invariant and $E$-stable homology theory $X$. Let $\Omega:kk\to kk$ be the inverse suspension; if $A,B\in\aha$, put 
\[
kk_n(A,B)=\hom_{kk}(j(A),\Omega^nj(B)),\quad kk(A,B)=kk_0(A,B).
\]
By \cite{kkwt}*{Theorem 8.2.1}, setting the first variable equal to the ground ring we recover Weibel's homotopy algebraic 
$K$-theory $KH$ \cite{weih}:
\[
kk_n(\ell,B)=KH_n(B).
\] 
Set
\[
KH^n(B):=kk_{-n}(B,\ell).
\]
Recall that a vertex $v\in E^0$ is \emph{regular} if it emits a nonzero finite number of 
edges and that it is \emph{singular} otherwise. Write $\reg(E)$ and $\sing(E)$ for the sets of regular and of singular edges. Let $A_E\in \Z^{\reg(E)\times E^0}$ be the matrix whose $(v,w)$ entry is the number of edges from $v$ to $w$ and let $I\in \Z^{E^0\times\reg(E)}$ be the matrix that results from the identity matrix upon removing the columns corresponding to singular vertices. It follows from \cite{abc} that if $KH_0(\ell)=\Z$, $KH_{-1}(\ell)$=0 and $E^0$ is finite, then for the reduced incidence matrix $A_E$ we have
\begin{equation}\label{intro:bf}
KH_0(L(E))=\coker(I-A_E^t).
\end{equation}
We show here (see Section \ref{sec:leinkk}) that, abusing notation, and writing $I$ for $I^t$,
\begin{equation}\label{intro:bf*}
KH^1(L(E))=\coker(I-A_E).
\end{equation}
For $n\ge 0$, let $L_n$ be the Leavitt path algebra of the graph with one vertex and $n$ loops. Thus $L_0=\ell$ and $L_1=\ell[t,t^{-1}]$. We prove the following structure theorem.

\begin{thm}\label{intro:struct} Assume that $KH_0(\ell)=\Z$ and $KH_{-1}(\ell)$=0. Let $E$ be a graph such that $E^0$ is finite.  
 Let $d_1,\dots,d_n$ , $d_i\backslash d_{i+1}$ be the invariant factors of the torsion group $\tau(E)=\tors(K_0(L(E))$, $s=\#\sing(E)$ and  $r=\rk(KH^1(L(E))$. Let $j:\aha\to kk$ be the universal excisive, homotopy invariant, $E$-stable homology theory. Then 
\[
j(L(E))\cong j(L_0^s\oplus L_1^r \oplus \bigoplus_{i=1}^nL_{d_i+1})
\]
\end{thm}
In particular, any unital Leavitt path algebra with trivial $KH_0$ is zero in $kk$. For example both $L_2$ and its Cuntz splice $L_{2^-}$ (\cite{ror}) are zero in $kk$. We also have the following corollary; here, and in any other statement which involves the image under $j$ of the Leavitt path algebras of finitely many graphs $E_1,\dots,E_n$, $j$ is understood to refer to the $\sqcup_{i=1}^nE_i$-stable $j$.

\begin{coro}\label{intro:dt}
Let $\ell$ be as in Theorem \ref{intro:struct}. The following are equivalent for graphs $E$ and $F$ with finitely many vertices.
\item[i)] $j(L(E)) \cong j(L(F))$.
\item[ii)] $KH_0(L(E)) \cong KH_0(L(F))$ and $KH^1(L(E)) \cong KH^1(L(F))$.
\item[iii)]$KH_0(L(E)) \cong KH_0(L(F))$ and $\#\sing(E) = \#\sing(F)$.
\end{coro}
\begin{proof} It is not hard to check, using \eqref{intro:bf} and \eqref{intro:bf*} (see Lemma \ref{lem:extk0}) that
the groups $KH_0(L(E))$ and $KH^1(L(E))$ have isomorphic torsion subgroups and that 
\begin{equation}\label{intro:sing}
\#\sing(E)=\rk KH_0(L(E))-\rk KH^1(L(E)).
\end{equation}
The corollary is immediate from this and the theorem above. 
\end{proof}

To put the above result in perspective, let us recall that E. Ruiz and M. Tomforde have shown in \cite{rt} that if $\ell$ is a field, $L(E)$ and $L(F)$ are simple and both $E$ and $F$ have infinite emitters, then condition iii) of the corollary holds if and only if $L(E)$ and $L(F)$ are Morita equivalent. 
Our result applies far more generally, but it is in principle weaker, since $kk$-isomorphic algebras need not be Morita equivalent. For example $L_2$ is not Morita equivalent to the $0$ ring. Observe also that the identity \eqref{intro:sing} helps us replace the graphic condition about $\#\sing$ by the purely $K$-theoretic or homological condition about $KH^1$. 

By \eqref{intro:bf*} and \cite{ck}*{Theorem 5.3}, when $E$ is finite and regular $KH^1(L(E))$ is isomorphic to the group of extensions of the  $C^*$-algebra of $E$ by the algebra of compact operators. We shall see presently that $KH^1(L(E))$ is also related to algebra extensions
\[
0\to M_\infty\to \cE\to L(E)\to 0.
\]
One can form an abelian monoid of homotopy classes of such extensions (see Section \ref{sec:homo}); we write $\cExt(L(E))$ for its group completion. When $E^0$ is finite and $E^1$ is countable, there is a natural map We show in Proposition \ref{prop:extkk} that, under the assumptions of Theorem \ref{intro:struct}, if in addition $E^1$ is countable and $E$ has no sources, then there is a natural surjection
\begin{equation}\label{intro:onto}
\cExt(L(E))\onto KH^1(L(E)).
\end{equation}
As another similarity with the operator algebra case, we prove (Corollary \ref{coro:uct}) that if $\ell$ and $E$ are as in Theorem \ref{intro:struct} and $R\in\aha$, then there is a short exact sequence
\begin{multline}\label{intro:uct}
0\to \ext^1_\Z(KH_0(L(E)), KH_{n+1}(R))\to kk_n(L(E), R)\overset{[KH_0,\gamma^*KH_1]}\lra\\
\Hom(KH_0(L(E)),KH_n(R))\oplus \Hom (\ker(I-A_E^t),KH_{n+1}(R))\to 0.
\end{multline}
Observe that, for operator algebraic $K$-theory, $K_1^{\mathrm{top}}(C^*(E))=\ker(I-A_E^t)$, so substituting $K^{\mathrm{top}}$ and $KK$ for $KH$ and $kk$ in \eqref{intro:uct} one obtains the usual UCT of \cite{roscho}*{Theorem 1.17}. Moreover, in Proposition \ref{prop:kun} we also prove an analogue of the K\"unneth theorem of \cite{roscho}*{Theorem 1.18}. 

Up to here in this introduction we have only discussed results that hold for $E$ with finitely many vertices and for $\ell$ such that $KH_0(\ell)=\Z$ and $KH_{-1}\ell=0$. With no hypothesis on $\ell$ we show that if $E$ and $F$ have finitely many vertices and $\theta\in kk(L(E),L(F))$ then
\begin{equation}\label{intro:kkiso}
\theta \text{ is an isomorphism }\iff KH_0(\theta) \text{ and } KH_1(\theta) \text{ are isomorphisms.}
\end{equation}
It is however not true that unital Leavitt path algebras with isomorphic $KH_0$ and $KH_1$ are $kk$-isomorphic, even when $\ell$ is a field (see Remark \ref{rem:ojokkiso}). Thus in view of Corollary \ref{intro:dt}, the pair $(KH_0,KH^1)$ is a better invariant of Leavitt path algebras than the pair $(KH_0,KH_1)$. 

Next let $\ell$ and $E$ be arbitrary and let $R\in\aha$. If $I$ is a set, write 

$$R^{(I)}=\bigoplus_{i\in I}R$$ 

for the algebra of finitely supported functions $I\to R$. Let $X:\aha\to\cT$ be an excisive, homotopy invariant, $E$-stable homology theory. Further assume that direct sums of at most $\# E^0$ summands exist in $\cT$ and that for any family of algebras $\{R_i: i\in I\}$ the natural map 
\[
\bigoplus_{i\in I}X(R_i)\to X(\bigoplus_{i\in I}R_i)
\]
is an isomorphism if $\#I\le \#E^0$. We prove in Theorem \ref{thm:basictrix} that there is a distinguished triangle in $\cT$ of the following form
\begin{equation}\label{intro:basictrix}
\xymatrix{X(R)^{(\reg(E))}\ar[r]^(0.55){I-A_E^t}&X(R)^{(E^0)}\ar[r]& X(L(E)\otimes R). }
\end{equation}
This applies, in particular, when we take $X=KH$, generalizing \cite{abc}*{Theorem 8.4}. Thus we get a long exact sequence
\begin{equation}\label{intro:khseq}
KH_{n+1}(L(E)\otimes R)^{(E^0)}\to KH_n(R)^{(\reg(E))}\overset{I-A_E^t}\lra KH_n(R)^{(E^0)}\to KH_n(L(E)\otimes R)
\end{equation}
When $R$ is regular supercoherent we may substitute $K$ for $KH$ in \eqref{intro:khseq}, generalizing \cite{abc}*{Theorem 7.6} (see Example \ref{ex:khtri}). 
Infinite direct sums are not known to exist in $kk$; however finite direct sums do exist, and $j$ does commute with them. Hence when $E^0$ is finite and $\ell$ is arbitrary, we may take $X=j$ above to obtain a distinguished triangle 
\begin{equation}\label{intro:basictri}
\xymatrix{j(R)^{\reg(E)}\ar[r]^(0.55){I-A_E^t}&j(R)^{E^0}\ar[r]& j(L(E)\otimes R).}
\end{equation}
This triangle is the basic tool we use to establish all the results on unital Leavitt path algebras mentioned above. 

\bigskip

The rest of this article is organized as follows. Some notations used throughout
the paper (in particular pertaining matrix algebras) are explained at the end of this
Introduction. In Section \ref{sec:homo} we recall some basic notions about algebraic homotopy, prove some elementary lemmas about it, and use them to define, for every pair of algebras $A$ and $R$ with $R$ unital, a group $\cExt(A,R)$ of virtual homotopy classes of extensions of $A$ by $M_\infty R$. In Section \ref{sec:kk} we recall some basic properties of $kk$ and quasi-homomorphisms.
Also, we prove in Proposition \ref{prop:sumaiotas} that if $\iota_i:A_i\to M_{S_i}A_i$ ($i\in I$) are corner inclusions, $S$ is an infinite set with $\# S\ge\#(\sqcup_iS_i)$ and $j:\aha\to kk$ is the universal excisive, homotopy invariant and $M_S$-stable homology theory, then the direct sum $\oplus_i\iota_i$ induces an isomorphism in $M_S$-stable $kk$, even if $I$-direct sums might not exist in $kk$.
Section \ref{sec:algcohn} is devoted to the characterization of the image under $j:\aha\to kk$ of the Cohn path algebra $C(E)$ of a graph $E$. The latter is related to Leavitt path algebra $L(E)$ by means of an exact sequence
\begin{equation}\label{intro:kcl}
0\to \cK(E)\to C(E)\to L(E)\to 0,
\end{equation}
where $\cK(E)$ is a direct sum of matrix algebras. The algebra $C(E)$ receives a canonical homomorphism $\phi:\ell^{(E^0)}\to C(E)$. We prove in Theorem \ref{thm:ftg} that the universal excisive, homotopy invariant, $E$-stable homology theory $j$ maps $\phi$ to an isomorphism
\begin{equation}\label{intro:phisoce}
j(\ell^{(E^0)})\cong j(C(E)).
\end{equation}
The proof uses quasi-homomorphisms, much in the spirit of Cuntz' proof of Bott periodicity for $C^*$-algebra $K$-theory. As a corollary we obtain that if $KH_0(\ell)=\Z$ and $E$ and $F$ are graphs, then for the universal excisive, homotopy invariant, $E\sqcup F$-stable homology theory $j$ we have
(Corollary \ref{coro:isoce}) 
\begin{equation}\label{intro:isoce}
j(C(E))\cong j(C(F))\iff K_0(C(E))\cong K_0(C(F))\iff \# E^0=\# F^0.
\end{equation} 
In Section \ref{sec:fundatri} we use Proposition \ref{prop:sumaiotas} to prove that 
$j(\ell^{(\reg(E))})\cong j(\cK(E))$. Putting this together with \eqref{intro:isoce} we get that, for arbitrary $\ell$ and $E$, the $kk$-triangle induced by \eqref{intro:kcl} is isomorphic to one of the form
\[
j(\ell^{(\reg(E))})\overset{f}{\to} j(\ell^{(E^0)})\to j(L(E)).
\]
We show in Proposition \ref{prop:basictri} that for each pair $(v,w)\in E^0\times \reg(E)$ the composite $\pi_vfi_w:\ell\to \ell $ induced by the inclusion at the $w$-summand and the projection onto the $v$-summand
is multiplication by the $(v,w)$-entry of $I-A_E^t$. We use this to prove \eqref{intro:basictrix} (Theorem \ref{thm:basictrix}). The exact sequence \eqref{intro:khseq}, the fact that $K$ can be substituted for $KH$ when $R$ is regular supercoherent, as well as triangle \eqref{intro:basictrix}, are deduced in Example \ref{ex:khtri}. The equivalence \eqref{intro:kkiso} is proved in Propostion \ref{prop:kkiso}. 
Beginning in Section \ref{sec:leinkk} we work under the Standing assumptions \ref{stan}, which are that $KH_{-1}\ell=0$ and $KH_0(\ell)=\Z$. The surjection \eqref{intro:onto} is established in Proposition \ref{prop:extkk}. The fact that
$KH_0(L(E))$ and $KH^1(L(E))$ have isomorphic torsion subgroups and the identity \eqref{intro:sing} are proved in Lemma \ref{lem:extk0}. Theorem \ref{intro:struct} and Corollary \ref{intro:dt} are Theorem \ref{thm:structure} and Corollary \ref{coro:kkk0ext}.
In Section \ref{sec:canofil} we introduce a descending filtration
$\{kk(L(E),R)^i:0\le i\le 2\}$ on $kk(L(E),R)$ for every algebra $R$ and every unital Leavitt path algebra $L(E)$ and compute the slices 
$kk(L(E),R)^i/kk(L(E),R)^{i+1}$ (Theorem \ref{thm:slice}). We use this to prove the universal coefficient theorem \eqref{intro:uct} in Corollary \ref{coro:uct} and the K\"unneth theorem in Proposition \ref{prop:kun}.

\begin{nota}\label{nota:intro} A commutative ground ring $\ell$ is fixed throughout the paper. All algebras, modules and tensor products are over $\ell$. If $A$ is an algebra and $X\subset A$ a subset, we write
$\mspan(X)$ and $\langle X\rangle$ for the $\ell$-submodule and the two-sided ideal generated by $X$. At the beginning of this Introduction we introduced, for a set $S$ and an algebra $A$, the algebra $M_SA$ of finitely supported $S\times S$-matrices. We write $M_S=M_S\ell$ and $\epsilon_{s,t}\in M_S$ for the matrix whose only nonzero entry is a $1$ at the $(s,t)$-spot $(s,t\in S)$. 
We also consider the algebra 
\[
\Gamma_S(R):=\{A:S\times S\to R\mid \#\supp A_{i,*},\#\supp A_{*,i}<\infty\} 
\] 
of those matrices which have finitely many nonzero coefficients in each row and column. If $\#S=n<\infty$, then $\Gamma_S=M_S=M_n$ is the usual matrix algebra.  We use special notation for the case $S=\N$; we write $M_\infty$ for $M_\N$ and $\Gamma$ for $\Gamma_\N$. Observe that $M_\infty R$ is an ideal of $\Gamma(R)$. Put
\begin{equation}\label{def:wagsus}
\Sigma(R)=\Gamma(R)/M_\infty R.
\end{equation}
The algebras $\Gamma(R)$ and $\Sigma(R)$ are Wagoner's \emph{cone} and \emph{suspension} algebras \cite{wagoner}. A \emph{$*$-algebra} is an algebra $R$ equipped with an involutive algebra homomorphism $R\to R^{op}$. For example $\ell$ is a $*$-algebra with trivial involution. If $R$ is a $*$-algebra, the conjugate matricial transpose makes both $\Gamma_S(R)$ and $M_S R$ into $*$-algebras. 
\end{nota}

\section{Homotopy and extensions}\label{sec:homo}

Let $\ell$ be a commutative ring. Let $\aha$ be the category of associative, not necessarily unital algebras over $\ell$.
If $B\in\aha$, we write $\ev_i:B[t]\to B$, $\ev_i(f)=f(i)$, $i=0,1$ for the evaluation map. Let $\phi_0,\phi_1:A\to B$ be two algebra homomorphisms; an \emph{elementary homotopy} from $\phi_0$ to $\phi_1$ is an algebra homomorphism $H:A\to B[t]$ such that
$\ev_0H=\phi_0$ and $\ev_1H=\phi_1$. We say that two algebra homomorphisms $\phi,\psi:A\to B$ are \emph{homotopic}, and write $\phi\simh\psi$, if for some $n\ge 1$ there is a finite sequence $\phi=\phi_0,\dots,\phi_n=\psi$ such that for each $0\le i\le n-1$ there is an elementary homotopy from $\phi_i$ to $\phi_{i+1}$. We write
\[
[A,B]=\hom_{\aha}(A,B)/\simh
\] 
for the set of homotopy classes of homomorphisms $A\to B$. 

\begin{lem}\label{lem:iotas}
Let $A$ be a ring. Then the maps $\iota_2,\iota_2':A\to M_2A$, $\iota_2(a)=\epsilon_{1,1}\otimes a$, $\iota_2'(a)=\epsilon_{2,2}\otimes a$ are homotopic. 
\end{lem}
\begin{proof}
Let $R=\tilde{A}$ be the unitalization. Consider the element 
\[
U(t)=\left[\begin{matrix}(1-t^2)& (t^3-2t)\\ t&(1-t^2)\end{matrix}\right]\in\GL_{2}R[t].
\]
Let $\ad(U(t)):R[t]\to R[t]$ be the conjugation map. Then $H=\ad(U(t))\iota_2:A\to M_2A[t]$, satisfies $\ev_0H=\iota_2$, $\ev_1H=\iota_2'$.  
\end{proof}

Let $A$ and $R$ be algebras, $\phi,\psi\in\hom_{\aha}(A,R)$ and $\iota_2:R\to M_2 R$, as in Lemma \ref{lem:iotas}. We say that $\phi$ and $\psi$ are \emph{$M_2$-homotopic}, and write $\phi\simh_{M_2}\psi$, if $\iota_2\phi\simh \iota_2\phi$. Put
\[
[A,R]_{M_2}=\hom_{\aha}(A,R)/\simh_{M_2}.
\]
Let $C$ be an algebra, $A,B\subset C$ subalgebras and $\inc_A:A\to C$, $\inc_B:B\to C$ the inclusion maps. Let $x,y\in C$ such that $yAx\subset B$ and $axya'=aa'$ for all $a,a'\in A$. Then 
\begin{equation}\label{map:adyx}
\ad(y,x):A\to B,\quad \ad(y,x)(a)=yax
\end{equation}
is a homomorphism of algebras, and we have the following.

\begin{lem}\label{lem:conju1}
Let $A, B, C$ and $x,y$ be as above. Then $\inc_B\ad(y,x)\simh_{M_2}\inc_A$. If moreover $A=B$ and $yA,Ax\subset A$, then $\ad(y,x)\simh_{M_2}\id_A$. 
\end{lem}
\begin{proof} Consider the diagonal matrices $\bar{y}=\diag(y,1), \bar{x}=\diag(x,1)\in M_2\tilde{C}$. One checks that $a\bar{x}\bar{y}a'=aa'$ for all $a,a'\in M_2A$. 
Hence $\phi:=\ad(\bar{y},\bar{x}):M_2A\to M_2C$ is a homomorphism. Moreover we have $\phi\iota_2=\iota_2\inc_B\ad(y,x)$ and $\phi\iota'_2=\iota'_2\inc_A$.
Thus applying Lemma \ref{lem:iotas} twice, we get
\[
\iota_2\inc_B\ad(y,x)\simh\iota'_2\inc_A\simh\iota_2\inc_A.
\]
This proves the first assertion. Under the hypothesis of the second assertion, $\phi$ maps $M_2A\to M_2A$, and we have $\phi\iota_2=\iota_2\ad(y,x)$ and
$\phi\iota'_2=\iota'_2$. The proof is immediate from this using Lemma \ref{lem:iotas}.
\end{proof}
 
A \emph{$C_2$-algebra} is a unital algebra $R$ together with a unital algebra homomorphism from the Cohn algebra $C_2$ to $R$. 
Equivalently, $R$ is a unital algebra together with elements $x_1,x_2,y_1,y_2\in R$ satisfying $y_ix_j=\delta_{i,j}$. 

If $R$ is a $C_2$-algebra  the map
\begin{equation}\label{map:boxplus}
\boxplus:R\oplus R\to R,\quad a\boxplus b=x_1ay_1+x_2ay_2
\end{equation}
is an algebra homomorphism. An \emph{infinite $C_2$-algebra} is a $C_2$-algebra together with an endomorphism $\phi:R\to R$ such that for all $a\in R$ we have 
\[
a\boxplus\phi(a)=\phi(a).
\]

\goodbreak

In the following lemma and elsewhere, if $M$ is an abelian monoid, we write $M^+$ for the group completion.

\begin{lem}\label{lem:boxplus}
Let $A$ be an algebra, $R=(R,x_1,x_2,y_1,y_2)$ a $C_2$-algebra, and $B\triqui R$ an ideal. Then \eqref{map:boxplus} induces an operation in $[A,B]_{M_2}$ which makes it into an abelian monoid whose neutral element is the zero homomorphism. If furthermore $R$ is an infinite $C_2$-algebra, then $[A,R]_{M_2}^+=0$. 
\end{lem}
\begin{proof}
By Lemma \ref{lem:conju1}, the homomorphisms $B\to B$, $b\mapsto x_iby_i$ ($i=0,1$) are $M_2$-homotopic to the identity. Hence to prove the first assertion, it suffices to show that \eqref{map:boxplus}  associative and commutative up to $M_2$-homotopy. This is straightforward from Lemma \ref{lem:conju1}, since all diagrams involved commute up to a map of the form \eqref{map:adyx}. The second assertion is clear. 
\end{proof}

\begin{ex}\label{ex:boxpis}
Any purely infinite simple unital algebra is a $C_2$-algebra, by \cite{agp}*{Proposition 1.5}. 
\end{ex}

\begin{ex}\label{ex:wagsus} If $R$ is a unital algebra, its cone $\Gamma(R)$ is an infinite $C_2$-algebra (\cite{wagoner}) and $\Sigma(R)$ is a $C_2$-algebra. For every algebra $R$, $\Gamma(R)\triqui\Gamma(\tilde{R})$ and $\Sigma(R)\triqui\Sigma(\tilde{R})$. By definition, we have an exact sequence 
\begin{equation}\label{seq:wag}
0\to M_\infty R\to \Gamma(R)\to \Sigma(R)\to 0.
\end{equation}   
\end{ex}

\begin{lem}\label{lem:wagoner}
Let $R$ be a unital algebra and let $\cE$ be an algebra containing $M_\infty R$ as an ideal. Then there exists a unique algebra homomorphism $\psi=\psi_\cE:\cE\to \Gamma(R)$
which restricts to the identity on $M_\infty R$.  
\end{lem}
\begin{proof} If $a\in \cE$ then for each $i,j\in\N$ there is a unique element $a_{i,j}\in R$ such that $(\epsilon_{i,i}\otimes 1)a(\epsilon_{j,j}\otimes 1)=\epsilon_{i,j}\otimes a_{i,j}$. One checks that $\psi:\cE\to \Gamma(R)$, $\psi(a)=(a_{i,j})$ satisfies the requirements of the lemma. 
\end{proof}

It follows from Lemma \ref{lem:wagoner} that if $R$ is unital then every exact sequence of algebras
\begin{equation}\label{seq:next}
0\to M_\infty R\to \cE\to A\to 0
\end{equation}
induces a homomorphism $\psi: A\to \Sigma(R)$ and that \eqref{seq:next} is isomorphic to the pullback along $\psi$ of \eqref{seq:wag}. Hence we may regard 
$[A,\Sigma(R)]_{M_2}$ as the abelian monoid of homotopy classes of all sequences \eqref{seq:next}. Put 
\begin{equation}\label{def:ext}
\cExt(A,R)=[A,\Sigma(R)]_{M_2}^+, \quad \cExt(A)=\cExt(A,\ell). 
\end{equation}
Observe that, by Lemma \ref{lem:boxplus}, any sequence \eqref{seq:next} which is split by an algebra homomorphism $A\to\cE$ maps to zero in $\cExt(A,R)$.

\section{Algebraic bivariant \emph{K}-theory}\label{sec:kk}

Let $\cT$ be a triangulated category and $\Omega$ the inverse suspension functor of $\cT$. A \emph{homology theory} with values in $\cT$ is a functor $X:\aha\to\cT$. An \emph{extension} of algebras is a short exact sequence of algebra homomorphisms
\begin{equation}\label{extensionE}
(E): 0\to A\to B\to C\to 0
\end{equation}
which is $\ell$-linearly split. We write $\cE$ for the class of all extensions. An \emph{excisive homology theory}
for $\ell$-algebras with values in $\cT$ consists of a functor
$X:\aha\to \cT$, together with a collection
$\{\partial_E:E\in\cE\}$ of maps $\partial_E^X=\partial_E\in\hom_{\cT}(\Omega
X(C), X(A))$ satisfying the compatibility conditions of \cite{kkwt}*{Section 6.6}. Observe that if $X:\aha\to\cT$ is excisive and $A,B\in\aha$, then the canonical map $X(A)\oplus X(B)\to X(A\oplus B)$ is an isomorphism. 
Let $I$ be a set. We say that a homology theory $X:\aha\to\cT$ is \emph{$I$-additive} if first of all direct sums of cardinality $\le \#I$ exist in $\cT$ and second of all the map
\[
\bigoplus_{j\in J}X(A_j)\to X(\bigoplus_{j\in J}A_j)
\]
is an isomorphism for any family of algebras $\{A_j:j\in J\}\subset\aha$ with $\#J\le\# I$. 

We say that the functor $X:\aha\to\cT$ is {\it homotopy invariant} if for every $A\in\aha$, $X$ maps the inclusion
$A\subset A[t]$ to an isomorphism. 

Let $S$ be a set, $s\in S$ and let 
\begin{equation}\label{map:iotas}
\iota_s:A\to M_SA,\quad \iota_s(a)=\epsilon_{s,s}\otimes a \qquad (A\in\aha). 
\end{equation}

Call $X$ {\it $M_S$-stable} if for every $A\in\aha$,
it maps $\iota_s:A\to M_S A$ to an isomorphism. This definition is independent of the element $s\in S$, by the argument of \cite{www}*{Lemma 2.2.4}. One can further show, using \cite{www}*{Proposition 2.2.6} and \cite{emathesis}*{Example 5.2.6} that if $S$ is infinite and $X$ is $M_S$-stable, and $T$ is a set such that $\#T\le\# S$, then $X$ is $T$-stable. 

\begin{defi}
Let $A,B\in\aha$. A  \emph{quasi-homomorphism} from $A$ to $B$ is a pair of homomorphisms
$\phi, \psi: A \to D\in\aha$, where $D$ contains $B$ as an ideal, such that
\[
\phi(a)-\psi(a)\in B\qquad (a\in A).
\]
We use the notation
\[
(\phi,\psi) : A \to D \rhd B.
\]

\end{defi}

Two algebra homomorphisms $\phi, \psi : A \to B$ are said to be {\it orthogonal}, in symbols $\phi\perp\psi$, if $\phi(x) \psi(y) = 0 = \psi(x) \phi(y)$ ($x,y\in A$). In this case, we will write $\phi   \perp \psi$. If $\phi   \perp \psi$ then $\phi   + \psi$ is an algebra homomorphism.  

\begin{prop} \label{prop:q-h} (\cite{crr}*{Proposition 3.3}) Let $X : \aha \to \tau $ be an excisive homology theory and let $(\phi,\psi): A \to D \rhd B$ be a quasi-homomorphism. Then, there is an induced map
\[ X(\phi,\psi): X(A) \to X(B) \] which satisfies the following naturality conditions:
\begin{enumerate}
\item $X(\phi,0)=X(\phi)$.
\item $ X(\phi,\psi) = - X(\psi,\phi)$.
\item If $(\phi_1, \psi_1)$ and $(\phi_2, \psi_2)$ are quasi-homomorphisms $A\to D\rhd B$ with $\phi_1   \perp \phi_2$ and $\psi_1   \perp \psi_2$, then $(\phi_1 + \phi_2, \psi_1 + \psi_2)$ is a quasi-homomorphism and 
$$ X(\phi_1 + \phi_2, \psi_1 + \psi_2) = X(\phi_1 , \psi_1 ) + X( \phi_2,  \psi_2). $$
\item $X(\phi, \phi) = 0$.
\item If $\alpha: C \to A$ is an $\ell$-algebra homomorphism, then
\[ X(\phi  \alpha,\psi  \alpha) = X(\phi,\psi)  X(\alpha). \]
\item If $\eta: D \to D'$ is an $\ell$-algebra homomorphism which maps $B$ into an
ideal $B' \lhd D'$, then
\[ X(\eta  \phi,\eta  \psi) = X(\eta|_B)  X(\phi,\psi).\]
\item Let $H =(H^+,H^-) : A \to D[t] \rhd B[t]$ with $ev_0 \circ H = (\phi^+,\phi^-)$ and  $ev_1 \circ H = (\psi^+,\psi^-)$. If, in addition, $X$ is homotopy invariant then 
$$ X(\phi^+,\phi^-)= X(\psi^+,\psi^-). $$
\item Let $(\psi, \varrho)$ be another quasi-homomorphism $A\to D\rhd B$. Then $( \phi, \varrho)$ is a quasi-homomorphism and
$$ X( \phi, \varrho) = X(\phi, \psi) + X(\psi, \varrho). $$
\end{enumerate}
\end{prop}

The excisive homology theories form a
category, where a homomorphism between the theories $X:\aha\to
\cT$ and $Y:\aha\to \cU$ is a triangulated functor $G:\cT\to\cU$ such
that $GX=Y$ and such that for every extension \eqref{extensionE} in $\cE$, the natural isomorphism
$\phi:G(\Omega X(C))\to \Omega Y(C)$ makes the following into a commutative diagram
\begin{equation*}
\xymatrix{G(\Omega X(C))\ar[r]^(.6){G(\partial^X_E)}\ar[d]_\phi&Y(A)\\
             \Omega Y(C).\ar[ur]_{\partial^Y_E}& }
\end{equation*}

In \cite{kkwt} a functor $j:\aha\to kk$ was defined which is an initial object in the full subcategory of those excisive homology theories which are
homotopy invariant and $M_\infty$-stable. It was shown in \cite{emathesis} that, for any fixed infinite set $S$,  by a  slight variation of the construction of \cite{kkwt} one obtains an initial object in the full subcategory of those excisive and homotopy invariant homology theories which are $M_S$-stable. Starting in the next section we shall fix $S$ and use $j$ and $kk$ for the universal excisive, homotopy invariant and $S$-stable homology theory and its target triangulated category. Moreover, we shall often omit $j$ from our notation, and say, for example, that an algebra homomorphism
is an isomorphism in $kk$ or that a diagram in $\aha$ commutes in $kk$ or that a sequence of algebra maps
\[
A\to B\to C
\]
is a triangle in $kk$ to mean that $j$ applied to the corresponding morphism, diagram or sequence is an isomorphism, a commutative diagram or a distinguished triangle. Also, since as explained above, in $kk$ the corner inclusion $\iota_s:A\to M_SA$ is independent of $s$, we shall simply write $\iota$ for 
$j(\iota_s)$.  

The loop functor $\Omega$ in $kk$ and its inverse have a concrete description as follows. Let $\Omega_1=t(t-1)\ell[t]$, 
$\Omega_{-1}=(t-1)\ell[t,t^{-1}]$. For $A\in\aha$ we have 
\begin{equation}\label{loop}
\Omega^{\pm 1} j(A)=j(\Omega_{\pm 1}\otimes A).
\end{equation}

\begin{ex}\label{ex:tenso}
Let $S$ be an infinite set and $j:\aha\to kk$ the universal homotopy invariant, excisive and $M_S$-stable homology theory. If $ R \in \aha$, then the functor $ j((-) \otimes R) : \aha \to kk$ is again a homotopy invariant, $M_S$-stable, excisive homology theory. Hence it gives rise to a triangulated functor $kk \to kk$. In particular, triangles in $kk$ are preserved by tensor products. Moreover, the tensor product induces a ``cup product"
\[
\cup:kk(A,B)\otimes kk(R,S)\to kk(A\otimes R,B\otimes S),\quad \xi\cup\eta=(B\otimes\eta)\circ(\xi\otimes R). 
\] 
\end{ex}

For $A,B\in\aha$ and $n\in\Z$, set
\begin{equation}\label{kkn}
kk_n(A,B)=\hom_{kk}(j(A),\Omega^nj(B)),\qquad kk(A,B)=kk_0(A,B).    
\end{equation}

The groups $kk_*(A,B)$ are the \emph{bivariant $K$-theory} groups of the pair $(A,B)$. Setting $A=\ell$ in \eqref{kkn} we recover the homotopy algebraic $K$-groups of Weibel \cite{weih}; there is a natural isomorphism (\cite{kkwt}*{Theorem 8.2.1}, \cite{emathesis}*{Theorem 5.2.20})
\begin{equation}\label{map:kkkh}
kk_*(\ell,B)\iso KH_*(B)\qquad (B\in\aha).
\end{equation}

\begin{rem}\label{rem:kknoli}
Even though $KH$ is $I$-additive for every set $I$, the universal functor $j:\aha\to kk$ is not known to be infinitely additive. 
\end{rem}

\begin{lem}\label{lem:homomat}
Let $\{A_i:i\in I\}\subset\aha$ be a family of algebras, $A=\bigoplus_{i\in I}A_i$, $T$ a set, $\ju:I\to T$ a function and $v\in T$. Then the homomorphism
\[
\iota_\ju:A\to M_TA,\quad \iota_\ju(\sum_i a_i)=\sum_{i\in I}\epsilon_{\ju(i),\ju(i)}\otimes a_i
\]
is homotopic to $\iota_v$.  
\end{lem}
\begin{proof} Because $(M_TA)[x]=\bigoplus_{i\in I}(M_TA_i[x])$, we may assume that $I$ has a single element, in which case the lemma follows using a rotational homotopy, as in the proof of Lemma \ref{lem:iotas}. 
\end{proof}
\begin{lem}\label{lem:biye=id}
Let $\{S_i:i\in I\}$ be a family of sets, $\sigma_i:S_i\to S_i$ an injective map, $(\sigma_i)_*:M_{S_i}\to M_{S_i}$, $(\sigma_i)_*(\epsilon_{s,t})=\epsilon_{\sigma_i(s),\sigma_i(t)}$ the induced homomorphism, $D=\bigoplus_{i\in I}M_{S_i}$, and $\sigma_*=\bigoplus_{i\in I}(\sigma_i)_*:D\to D$. 
If $X:\aha\to\cT$ is $M_2$-invariant, then $X(\sigma_*)$ is the identity map. 
\end{lem}
\begin{proof} The map $\sigma_i$ induces an $\ell$-module homomorphism $\ell^{(S_i)}\to\ell^{(S_i)}$ whose matrix $[\sigma_i]$ is an element of the ring $\Gamma_{S_i}$ of \ref{nota:intro}. Let $[\sigma_i]^*$ be the transpose matrix; we have $[\sigma_i]^*[\sigma_i]=1$, and $(\sigma_i)_*(a)=[\sigma_i]a[\sigma_i^*]$ $(a\in M_{S_i})$. Hence for $[\sigma]=\bigoplus_{i\in I}[\sigma_i]\in R=\bigoplus_{i\in I}\Gamma_{S_i}$, we have $\sigma_*(a)=[\sigma]a[\sigma]^*$. Since $D\triqui R$, $X(\sigma_*)$ is the identity by \cite{www}*{Proposition 2.2.6}.
\end{proof}

\begin{prop}\label{prop:sumaiotas} Let $\{S_i:i\in I\}$ be a family of sets, $v_i\in S_i$ and $S=\coprod_{i\in I}S_i$. Let $f=\bigoplus_{i\in I}\iota_{v_i}:\ell^{(I)}\to \oplus_{i\in I}M_{S_i}$. Let $T$ be an infinite set with $\#T\ge \#S$. Let $j:\aha\to kk$ be the universal excisive, homotopy invariant and $M_T$-stable homology theory. Then $j(f)$ is an isomorphism.
\end{prop}
\begin{proof}
Put $D=\bigoplus_{i\in I}M_{S_i}$. Let $\inc:D\to M_S\ell^{(I)}$ be the inclusion. By Lemma \ref{lem:homomat}, the composite $\inc f$ equals the canonical inclusion $\iota$ in $kk$. Next let $g=(M_Sf)\inc:D\to M_SD$. We have $g(\epsilon_{\alpha,\beta})=\epsilon_{\alpha,\beta}\otimes \epsilon_{v_i,v_i}$ ($\alpha,\beta\in S_i$). For each $i\in I$ extend the coordinate permutation map $S_i\times\{v_i\}\to \{v_i\}\times S_i$, to a bijection $\sigma_i:S\times S_i\to S\times S_i$, and let $(\sigma_i)_*$ be the induced automorphism of $M_S M_{S_i}\cong M_{S\times S_i}$. Consider the automorphism $\sigma_*=\bigoplus_{i\in I}(\sigma_i)_*:M_SD\to M_SD$; by Lemmas \ref{lem:homomat} and \ref{lem:biye=id}, $\iota=j(\sigma_*g)=j(g)$. From what we have just seen and Example \ref{ex:tenso}, in $kk$ the following diagram commutes and its horizontal arrows are isomorphisms.
\[
\xymatrix{\ell^{(I)}\ar[d]_{f}\ar[r]^{\iota}& M_S\ell^{(I)}\ar[d]_{M_Sf}\ar[r]^{M_S\iota}&M_SM_S\ell^{(I)}\ar[d]^{M_SM_Sf}\\
          D\ar[r]^{\iota}\ar[ur]^{\inc}& M_SD\ar[r]_{M_S\iota}\ar[ur]^{M_S\inc}&M_SM_SD} 
\]
It follows that $M_Sf$ and $f$ are isomorphisms in $kk$.
\end{proof}

\section{Cohn algebras and \emph{kk}}\label{sec:algcohn}

A directed \emph{graph} is a quadruple $E=(E^0, E^1, r, s)$ where $E^0$ and $E^1$ are the sets of vertices and edges, and $r$ and $s$ are the \emph{range} and \emph{source} functions $E^1\to E^0$. We call $E$ \emph{finite} if both $E^0$ and $E^1$ are finite. A vertex $v \in E^0$ is a \emph{sink} if $s^{-1}(v) = \emptyset$ and is an \emph{infinite emitter} if $s^{-1}(v)$ is infinite. A vertex $v$ is \emph{singular} if it is either a sink or an infinite emitter; we call $v$ \emph{regular} if it is not singular.  A vertex $v \in E^0$ is a {\it source} if $r^{-1}(v) = \emptyset$. We write $\sink(E)$, $\inf(E)$ and $\sour(E)$ for the sets of sinks, infinite emitters, and sources, and $\sing(E)$ and $\reg(E)$ for those of singular and of regular vertices. 

A finite {\it path}  $\mu$ in a graph $E$ is a sequence of edges $\mu = e_1 \dots e_n$ such that $r(e_i) = s(e_{i+1})$ for $i = 1, \dots, n-1$. In this case $|\mu|:=n$ is the \emph{length} of $\mu$. We view the vertices of $E$ as paths of length $0$. Write $\path(E)$ for the set of all finite paths in $E$. The range and source functions $r,s$ extend to  $\path(E) \to E^0$ in the obvious way. An edge $f$ is an {\it exit} for a path $\mu = e_1 \dots e_n$ if there exist $i$ such that $s(f) = s(e_i) $ and $f \neq e_i$. A path $\mu = e_1 \dots e_n$ with $n \ge 1$ is a {\it closed path} at $v$ if $s(e_1) = r(e_n) = v$. A closed path $\mu = e_1 \dots e_n$ at $v$ is a {\it cycle} at $v$ if $s(e_j) \neq s(e_i) $ for $i \neq j$.

The {\it Cohn path algebra} $C(E)$ of a graph $E$ is the quotient of the free associative $\ell$-algebra generated by the set $ E^0 \cup E^1 \cup \{ e^* \mid e \in E^1\}$, subject to the relations:
\begin{enumerate}
    \item[(V)] $ v \cdot w = \delta_{v,w} v$.
    \item [(E1)] $s(e) \cdot e = e = e \cdot r(e)$.
    \item [(E2)] $r(e) \cdot e^* = e^* = e^* \cdot s(e)$.
    \item [(CK1) ]  $ e^* \cdot f = \delta_{e,f} \ r(e)$.
\end{enumerate}
The algebra $C(E)$ is in fact a $*$-algebra; it is equipped with an involution $*:C(E)\to C(E)^{op}$ which fixes vertices and maps $e\mapsto e^*$ $(e\in Q^1$).  
Condition $V$ says that the vertices of $E$ are orthogonal idempotents in $C(E)$. Hence the subspace generated by $E^0$ is a subalgebra of $C(E)$, isomorphic to the algebra $\ell^{(E^0)}$ finitely supported functions $E^0\to \ell$. For $v\in E^0$, let $\chi_v\in\ell^{(E^0)}$ be the characteristic function of $\{v\}$.
We have a monomorphism
\begin{equation}\label{map:phi}
\varphi: \ell^{(E^0)}\to C(E), \quad \varphi(\chi_v)=v.
\end{equation}
Observe that if $E^0$ is finite, then $\ell^{(E^0)}=\ell^{E^0}$ is the algebra of all functions $E^0\to \ell$.

We shall say that a homology theory is \emph{$E$-stable} if it is stable with respect to a set of cardinality $\#(E^0\coprod E^1\coprod\N)$.  

The main result of this section is the following theorem.

\begin{teo}\label{thm:ftg} Let $\varphi$ be the algebra homomorphism \eqref{map:phi} and let $j:\aha\to kk$ be the universal excisive, homotopy invariant and $E$-stable homology theory. Then $j(\varphi)$ is an isomorphism.
\end{teo}

\begin{coro}\label{coro:isoce}
Let $E$ and $F$ be graphs and $j:\aha\to kk$ the universal excisive, homotopy invariant and $E\sqcup F$-stable homology theory. Assume that 
$KH_0(\ell)\cong \Z$. Then $C(E)$ and $C(F)$ are isomorphic in $kk$ if and only if $\# E^0=\# F^0$. 
\end{coro}
\begin{proof} By Theorem \ref{thm:ftg}, $C(E)$ and $C(F)$ are isomorphic in $kk$ if and only if $\ell^{(E^0)}$ and $\ell^{(F^0)}$ are. If $\# E^0=\# F^0$ then $\ell^{(E^0)}$ and $\ell^{(F^0)}$ are isomorphic in $\aha$, and therefore also in $kk$. Assume conversely that $\ell^{(E^0)}$ and $\ell^{(F^0)}$ are isomorphic in $kk$. Then in view of \eqref{map:kkkh} and of the hypothesis that $KH_0(\ell)\cong\Z$, we have $\# E^0=\# F^0$.
\end{proof}

The proof of Theorem \ref{thm:ftg} is organized in four parts, with three lemmas interspersed. First we need some preliminaries.

 Associate an element $m_v\in C(E)$ to each $v\in E^0\setminus\inf(E)$ as follows
  \[
  m_v = \begin{cases}
  \sum_{e \in s^{-1}(v)  } ee^* & \qquad \text{if } v \in \reg(E)  \\
  0 & \qquad \text{if } v \in \sour(E).  \end{cases} 
  \]
Observe that $m_v$ satisfies the following identities:
\begin{equation}\label{mv}
m_v=m_v^*,\ \ m_v^2=m_v,\ \ m_vw=\delta_{w,v}m_v,\ \ m_ve=\delta_{v,s(e)}e\quad (w\in E^0,\ \ e\in E^1).
\end{equation}
Let $C^m(E)$ be the $*$-algebra obtained from $C(E)$ by formally adjoining an element $m_v$ for each $v\in\inf(E)$ subject to the identities \eqref{mv}. We have a canonical $*$-homomorphism
\begin{equation}\label{map:can}
\can:C(E)\to C^m(E).
\end{equation}
Let $\cP=\cP(E)$. For $v\in E^0$, set
\begin{equation}\label{parribajo}
\cP_v=\{\mu \in \path(E) \mid r(\mu) = v \}, \quad \cP^v=\{\mu\in\cP\mid s(\mu)=v\}.
\end{equation}
Let $\Gamma_\cP$ be the ring introduced in \ref{nota:intro}. Using the notation \eqref{parribajo} in the summation indexes, define a $*$-homomorphism
\begin{gather}\label{map:rho}
\rho:C^m(E)\to \Gamma_\cP,\\
\rho(v)=\sum_{\alpha\in\cP^v}\epsilon_{\alpha,\alpha},\quad \rho(e)=\sum_{\alpha\in \cP^{r(e)}}\epsilon_{e\alpha,\alpha},\quad (v\in E^0, e\in E^1)\nonumber\\
\rho(m_w)=\sum_{\alpha\in\cP^{w},|\alpha|\ge 1}\epsilon_{\alpha,\alpha}\quad (w\in\inf(E)).\nonumber
\end{gather}
\begin{lem}\label{lem:rhomono}
The maps \eqref{map:can} and \eqref{map:rho} are monomorphisms.
\end{lem}
\begin{proof} It is well-known that the set 
\[
\cB_1=\{\alpha\beta^*\mid \alpha,\beta\in\cP, r(\alpha)=r(\beta)\}
\]
is a basis of $C(E)$ (\cite{libro}*{Proposition 1.5.6}). Set
\[
\cB_2=\{\alpha m_v\beta^*\mid \alpha,\beta\in \cP_v, v\in\inf(E)\}.
\]
It follows from \eqref{mv} that $\cB=\cB_1\cup\cB_2$ generates $C^m(E)$ as an $\ell$-module. It is clear that $\rho$ is injective on $\cB$; hence it suffices to show that the set $\rho(\cB)\subset\Gamma_\cP$ is $\ell$-linearly independent. Let $\cF\subset\cB$ be a finite set and $c:\cF\to \ell\setminus\{0\}$ a function such that 
\[
\sum_{x\in\cF}c_x x=0.
\]
Let $Q=\{(\alpha,\beta)\in\cP^2\mid r(\alpha)=r(\beta)\}$; give $Q$ a partial order by setting $(\alpha,\beta)\ge (\alpha',\beta')$ if and only if there exists $\theta\in\cP_{r(\alpha)}$ such that $\alpha'=\alpha\theta$, $\beta'=\beta\theta$. Let $f:\cB\to Q$, $f(\alpha\beta^*)=(\alpha,\beta)$, $f(\alpha m_v\beta^*)=(\alpha,\beta)$. Assume that $\cF\ne\emptyset$. Then $f(\cF)$ has a maximal element $(\alpha,\beta)$. If $\alpha\beta^*\in\cF$, then
$\rho(\alpha\beta^*)$ is the only matrix in $\rho(\cF)$ whose $(\alpha,\beta)$ entry is nonzero. Thus $c_{\alpha\beta^*}=0$, a contradiction. Hence $v=r(\alpha)\in\inf(E)$, $\alpha\beta^*\notin\cF$ and $\alpha m_v\beta^*\in \cF$. Then $f(\cF\setminus\{\alpha m_v\beta^*\})$ contains only finitely many elements of the form $(\alpha e,\beta e)$ with $e\in s^{-1}(v)$. However $\rho(\alpha m_v\beta^*)_{\alpha e,\beta e}=1$ for every $e\in s^{-1}(v)$. Thus $c_{\alpha m_v\beta^*}=0$, which again is a contradiction. Hence $\cF$ must be empty; this concludes the proof. 
\end{proof}

\begin{rem} By Lemma \ref{lem:rhomono} we may identify  $C^m(E)$ with its image in $\Gamma_\cP$. Under this identification, the formula
\[
m_v= \sum_{e \in s^{-1}(v)  } ee^* 
\]
holds for every $v\in E^0$. 
\end{rem}

\noindent{\it Proof of Theorem \ref{thm:ftg}, part I}.  
 Set 
\begin{equation}\label{qv}
C^m(E)\owns q_v=v-m_v\quad (v\in E^0).
\end{equation}  
Consider the following ideals of $C^m(E)$ 
\begin{equation}\label{losck}
\cK(E)=\langle q_v \mid  v \in \reg(E) \rangle \subset\hat{\cK}(E)=\langle q_v  \mid v \in E^0\rangle. 
\end{equation}
One checks, using \cite{libro}*{Proposition 1.5.11} that the maps 
\[
M_{\cP_v}\to \hat{\cK}(E),\quad \epsilon_{\alpha,\beta}\mapsto \alpha q_v\beta^*\quad (v\in E^0)
\]
assemble to an isomorphism 
\begin{equation}\label{map:isokj}
\bigoplus_{v\in E^0} M_{\cP_v}\iso \hat{\cK}(E).
\end{equation} 
Observe that \eqref{map:isokj} restricts to an isomorphism
\begin{equation}\label{map:isok}
\bigoplus_{v\in\reg(E)} M_{\cP_v}\iso\cK(E). 
\end{equation}
Let $\hat{\iota} : \ell^{(E^0)} \to \hat{\cK}(E)$ be the homomorphism that sends the canonical basis element $\chi_v$ to $q_v$ and let $\xi:C(E)\to C^m(E)$ be the $*$-homomorphism determined by
\[
    \xi(v)=m_v,\quad \xi(e)=em_{r(e)}.
\]
One checks that $(\can ,\xi)$ is a quasi-homomorphism $C(E) \to C^m(E) \rhd \hat{\cK}(E)$. From the equality $ \can\varphi = \xi \varphi + \hat{\iota}$ and items (1), (3), (4) and (5) of Proposition \ref{prop:q-h}, it follows that 
\[
j(\can,\xi) j(\varphi) = j(\can\varphi ,\xi \varphi)= j(\xi \varphi + \hat{\iota} ,\xi \varphi) = j( \xi \varphi, \xi \varphi) + j(\hat{\iota},0 ) = j(\hat{\iota}).
\]
By Proposition \ref{prop:sumaiotas}, $\hat{\iota}$ is an isomorphism in $kk$. Hence
\[
j(\hat{\iota})^{-1} j(\can,\xi) j(\varphi) = 1_{j(\ell^{(E^0)})}.
\] 
It remains to show that 
\begin{equation}\label{vueltacohn}
j(\varphi)j(\hat{\iota})^{-1} j(\can,\xi)  = 1_{j(C(E))}.
\end{equation}
Let $\cP=\cP(E)$; consider the algebra $M_\cP$ of finite matrices indexed by $\cP$. Let $\hat{\varphi}:\hat{\cK}(E)\to M_{\cP}C(E)$ be the homomorphism that sends  $\alpha q_v \beta^\ast$ to $ \epsilon_{\alpha,\beta} \otimes v$, where $\epsilon_{\alpha,\beta}$ is the matrix unit. We shall need a twisted version $\hat{\iota}_\tau$ of $\hat{\iota}$; this is the $*$-homomorphism
\begin{equation}\label{map:iotatau}
\hat{\iota}_\tau : C(E)\to  M_{\cP}C(E),\quad \hat{\iota}_\tau(v)=\epsilon_{v,v},\quad \hat{\iota}_\tau(e)=\epsilon_{s(e),r(e)} \otimes e
\qquad (v\in E^0,e\in E^1).
\end{equation}
We have a commutative diagram 
\begin{equation}\label{diag:primero}
\xymatrix{ \ell^{(E^0)} \ar[d]_{\varphi}\ar[r]^{\hat{\iota}}& \hat{\cK}(E)\ar[d]^{\hat{\varphi}}\\
 C(E) \ar[r]_{\hat{\iota}_\tau}& M_{\cP}C(E)}
\end{equation}

\begin{lema}\label{lem:sameiso} Let $\alpha\in \cP$ and let $\iota_\alpha:C(E) \to M_{\cP}C(E) $ as in \eqref{map:iotas}. Then $\iota_\alpha$  and the map $\hat{\iota}_\tau$ of \eqref{map:iotatau} induce the same isomorphism in $kk$.
\end{lema}
\begin{proof}
Because $j$ is $E$-stable, it is $M_\cP$-stable, whence $\iota_\alpha$ is an isomorphism and does not depend on $\alpha$. Hence we may and do assume that $\alpha=w\in E^0$. Because $j$ is homotopy invariant, it is enough to find a polynomial homotopy between $\iota_w$ and $\hat{\iota}_\tau$. For each $v \in E^0 \backslash \{w\}$ set
\begin{align*}
A_v =& [(1-t^2)\epsilon_{w,w}+ (t^3-2t) \epsilon_{w,v} + t \epsilon_{v,w}+ (1-t^2) \epsilon_{v,v} ] \otimes v,\\
B_v =& [(1-t^2)\epsilon_{w,w}+ (2t-t^3) \epsilon_{w,v} - t \epsilon_{v,w}+ (1-t^2) \epsilon_{v,v} ] \otimes v, \quad A_w = \epsilon_{w,w} \otimes w = B_w. 
\end{align*}
The desired homotopy is the homomorphism $H:C(E)\to M_\cP C(E)[t]$ defined by 
\[
H(v)=A_v (\epsilon_{v,v} \otimes v) B_v,\quad H(e)=A_{s(e)}(\epsilon_{s(e),r(e)}\otimes e)B_{r(e)},\quad H(e^*)=A_{r(e)}(\epsilon_{r(e),s(e)}\otimes e^*)B_{s(e)}.
\]
\end{proof}

\noindent{\it Proof of Theorem \ref{thm:ftg}, part II}. Let 
\[
M_{\cP}C(E)\supset \fA=\mspan\{\epsilon_{\gamma, \delta} \otimes \alpha \beta^\ast \mid s(\alpha) = r(\gamma), s(\beta) = r(\delta), r(\alpha) = r(\beta)\}.
\]
One checks that $\fA$ is a subalgebra containing the images of both $\hat{\iota}_\tau$ and $\hat{\varphi}$. From the commutative diagram \ref{diag:primero} we obtain, by corestriction, another commutative diagram
\begin{equation}\label{diag:segundo}
\xymatrix{ \ell^{(E^0)} \ar[d]_{\varphi}\ar[r]^{\hat{\iota}}& \hat{\cK}(E)\ar[d]^{\hat{\varphi}}\\
 C(E) \ar[r] & \fA}
\end{equation}
By Lemma \ref{lem:sameiso}, the bottom arrow of \eqref{diag:segundo} is a monomorphism in $kk$. We shall abuse notation and write $\hat{\iota}_\tau$ for the latter map. 

Let $\widetilde{C}^m(E)$ be the unitalization; put $R=\Gamma_{\cP}\widetilde{C}^m(E)$. Consider the homomorphism $\rho'=\rho\otimes 1:C(E)\to R$. One checks that the subalgebra $\fA\subset R$ is closed under both left and right multiplication by elements in the image of $\rho'$. We can thus form the semi-direct product $C^m(E) \ltimes \fA=C^m(E) \ltimes_{\rho'} \fA$. As an $\ell$-module, 
 $C^m(E) \ltimes \fA$ is just $C^m(E) \oplus \fA$. Multiplication is defined by the rule 
$$ (r,x) \cdot (s,y) = (rs, \rho'(r)x+y\rho'(s)+xy).$$
Let $J$ be the ideal in $C^m(E) \ltimes \fA$ generated by the elements $(\alpha q_v \beta^\ast, -\epsilon_{\alpha,\beta}\otimes v)$ with $v = r(\alpha) = r(\beta)$. One checks that 
\[
J =\mspan\{(\alpha q_v \beta^\ast, -\epsilon_{\alpha,\beta}\otimes v) \ : \ v = r(\alpha) = r(\beta) \}.
\] 
Set 
\[
D=(C^m(E) \ltimes \fA)/J.
\]

\begin{lem}\label{lem:Ideal} The composite of the inclusion and projection maps $\fA=0\rtimes \fA\subset C^m(E)\ltimes \fA\to D$ is injective.
\end{lem}
\begin{proof}
It follows from \eqref{map:isokj} that there is an injective homomorphism 
\[
\ju:\hat{\cK}(E)\to \fA,\quad \ju(\alpha q_v\beta^*)=\epsilon_{\alpha,\beta}\otimes v\quad (r(\alpha)=r(\beta)=v).
\] 
Let $\inc:\hat{\cK}(E)\to C^m(E)$ be the inclusion. Observe that $J$ is the image of the map $\inc\rtimes(-\ju):\hat{\cK}(E)\to C^m(E)\rtimes \fA$. In particular, the projection $\pi:C^m(E)\ltimes \fA\to C^m(E)$ is injective on $J$. It follows that $J\cap (0\rtimes \fA)=0$; this finishes the proof. 
\end{proof}

\noindent{\it Proof of Theorem \ref{thm:ftg}, part III.} By Lemma \ref{lem:Ideal}, we may regard $\fA$ as an ideal of $D$. Let $\Upsilon:C^m(E)\to D$ be the composite of the inclusion $C^m(E)\subset C^m(E)\rtimes \fA$ and the projection $C^m(E)\rtimes \fA\to D$. We may embed diagram \eqref{diag:segundo} into a commutative diagram

\begin{equation}\label{diag:tercero}
\xymatrix{ \ell^{(E^0)} \ar[d]_{\varphi}\ar[r]^{\hat{\iota}}& \hat{\cK}(E)\ar[d]^{\hat{\varphi}}\ar[r]& C^m(E) \ar[d]^{\Upsilon}\\
 C(E) \ar[r]^{\hat{\iota}_\tau} & \fA \ar[r] & D}
\end{equation}

Let $\psi_0=\Upsilon\can$, $\psi_1 = \Upsilon\xi$. Note that $\psi_1 \perp \hat{\iota}_\tau$, so $\psi_{1/2} = \psi_1 + \hat{\iota}_\tau$ is an algebra homomorphism. We have quasi-homomorphisms 
\[
(\psi_0,\psi_1) ,(\psi_0,\psi_{1/2}),(\psi_{1/2},\psi_1) : C(E) \to D \rhd \fA.
\]

\begin{lema}\label{lem:pepe}
The quasi-homorphism $(\psi_0,\psi_{1/2})$ induces the zero map in $kk$.
\end{lema}
\begin{proof}
Let $H^+: C(E) \to D[t]$ be the algebra homorphism determined by setting
\begin{gather*}
    H^+(v)=(v,0),\quad H^+(e)= (e m_{r(e)},0 ) + (1-t^2) (0, \epsilon_{s(e),r(e)} \otimes e) + t (0, \epsilon_{e,r(e)} \otimes r(e))\\
    H^+(e^\ast)=( m_{r(e)} e^\ast ,0 ) + (1-t^2) (0, \epsilon_{r(e),s(e)} \otimes e^\ast) + (2t-t^3)  (0, \epsilon_{r(e),e} \otimes r(e))\end{gather*}
for $v\in E^0$ and $e\in E^1$. It is a matter of calculation to show that $H^+$ a homotopy between $\psi_0$ and $\psi_{1/2}$, and that  $(H^+,\psi_{1/2}) : C(E) \to D[t]\rhd \fA[t]$ is a homotopy between $(\psi_0,\psi_{1/2})$ and $(\psi_{1/2},\psi_{1/2})$. Hence by item (7) of Proposition \ref{prop:q-h}, we obtain
$$ j(\psi_0,\psi_{1/2}) = j(\psi_{1/2},\psi_{1/2})  = 0$$
as wanted.
\end{proof}

\noindent{\it Proof of Theorem \ref{thm:ftg}, conclusion}. Using the commutativity of diagram \eqref{diag:tercero} and items (6), (8) and (1) of Proposition \ref{prop:q-h} and  Lemma \ref{lem:pepe} we have
$$ j(\hat{\varphi}) j(\can,\xi) = j(\psi_0,\psi_1) =j(\psi_0,\psi_{1/2})+j(\psi_{1/2},\psi_1) = j(\hat{\iota}_\tau).  $$
On the other hand
$$  j(\hat{\varphi}) j(\can,\xi)= j(\hat{\iota}_\tau) j(\varphi) j(\hat{\iota})^{-1} j(\can,\xi).$$
Hence
\[j(\hat{\iota}_\tau) = j(\hat{\iota}_\tau)j(\varphi)  j(\hat{\iota})^{-1} j(1,\xi)\]

Since $j(\hat{\iota}_\tau)$ is a monomorphism, 
this implies that $$1_{j(C(E))} =  j(\varphi) j(\hat{\iota})^{-1} j(1,\xi).$$ 

This finishes the proof.\qed

\section{The Leavitt path algebra and a fundamental triangle}\label{sec:fundatri}

Let $E$ be a graph; for $v\in E^0$ let $q_v\in C(E)$ be the element \eqref{qv}. The \emph{Leavitt path algebra} $L(E)$ is the quotient of $C(E)$ modulo the relation 
\[
(CK2)\quad q_v=0 \quad (v\in \reg(E)).
\]
In other words, for the ideal $K(E)\triqui C(E)$ of \eqref{losck}, we have a short exact sequence
\begin{equation}\label{ses:cele}
0\to \cK(E)\to C(E)\to L(E)\to 0.
\end{equation}
It follows from \cite{libro}*{Proposition 1.5.11} that the sequence \eqref{ses:cele} is $\ell$-linearly split, and is thus an algebra extension in the sense of Section \ref{sec:kk}.

The \emph{adjacency matrix} $A'_E \in \Z^{((E^0\setminus\inf(E)) \times E^0)}$ is the matrix whose entries are given by 
\[
(A'_E)_{v,w} = \#\{e \in E^1 \ : \ s(e) = v \text{ and } r(e) = w\}.
\]
The \emph{reduced adjacency matrix} is the matrix $A_E\in\Z^{(\reg(E)) \times E^0)}$ which results from $A_E$ upon removing the rows corresponding to sinks.
We also consider the matrix
\[
I\in \Z^{(E^0\times\reg(E))},\quad I_{v,w}=\delta_{v,w}.
\]
\begin{prop}\label{prop:basictri}
Let $j:\aha\to kk$ be as in Theorem \ref{thm:ftg}. 
\item[i)] There is a distinguished triangle in $kk$
\begin{equation}\label{tri:basic1}
\xymatrix{j(\ell^{(\reg(E))})\ar[r]^f& j(\ell^{(E^0)})\ar[r]& j(L(E)).}
\end{equation}
\item[ii)] Let $\chi_v:\ell\to \ell^{(\reg(E))}$ be the inclusion in the $v$-summand  and let $c_v\in \Z^{(E^0)\times \{v\}}$ be the $v$-column of the matrix
$I-A_E^t$ $(v\in\reg(E))$. Under the isomorphism \eqref{map:kkkh}, the composite 
$f j(\chi_v)$ corresponds to the map
\[
1\otimes c_v:KH_0(\ell)\to KH_0(\ell)\otimes \Z^{(E^0)} 
\] 
\end{prop}
\begin{proof} Consider the map $q:\ell^{(\reg(E))}\to \cK(E)$, $q(\chi_v)=q_v$. In view of \eqref{map:isok}, $j(q)$ is an isomorphism by Proposition \ref{prop:sumaiotas}. By Theorem \ref{thm:ftg}, the map $j(\phi)$ is an isomorphism. Hence the $kk$-triangle induced by \eqref{ses:cele} is isomorphic to the triangle \eqref{tri:basic1} where for the inclusion $\inc:\cK(E)\subset C(E)$, we have $f=j(\phi)^{-1}j(\inc)j(q)$. This proves i). To prove ii), fix $v\in \reg(E)$ and consider the elements $q_v, m_v$ and $ee^*\in C(E)$ ($e\in E^1$, $s(e)=v$). As the latter elements are idempotent, we regard them as homomorphisms $\ell\to C(E)$. In particular, $q_v=\inc q\chi_v$. Because $q_v\perp m_v$ and $v=q_v+m_v$, $j(q_v)=j(v)-j(m_v)$. On the other hand, by (CK1), $j(m_v)=\sum_{s(e)=v} j(r(e))$. Summing up, $q_v=j(v)-\sum_{s(e)=v} j(r(e))$; this proves ii).
\end{proof}

\begin{thm}\label{thm:basictrix}
Let $X:\aha\to\cT$ be an excisive, homotopy invariant, $E$-stable and $E^0$-additive homology theory and let $R\in\aha$. Then \eqref{tri:basic1} induces a triangle in $\cT$
\[
\xymatrix{X(R)^{(\reg(E))}\ar[r]^{I-A_E^t}& X(R)^{(E^0)}\ar[r]& X(L(E)\otimes R)}.
\]
\end{thm}
\begin{proof} Tensoring the triangle \eqref{tri:basic1} by $R$ yields another triangle in $kk$, by Example \ref{ex:tenso}. By the universal property of $j$, applying $X$ to the latter triangle gives a distinguished triangle in $\cT$. Now apply Proposition \ref{prop:basictri} (ii) and the $E^0$-additivity hypothesis on $X$ to finish the proof. 
\end{proof}

\begin{ex}\label{ex:khtri}
Theorem \ref{thm:basictrix} applies to $X=KH$ and arbitrary $E$, generalizing \cite{abc}*{Theorem 8.4} from the row-finite to the general case. Recall a ring $A$ is \emph{$K_n$-regular} if for every $m\ge 1$, the inclusion $A\to A[t_1,\dots,t_m]$ induces an isomorphism $K_n(R)\to K_n(R[t_1,\dots,t_m])$. We call $A$ \emph{$K$-regular} if it is $K_n$-regular for all $n$. By \cite{weih}*{Proposition 1.5}, the canonical map $K(A)\to KH(A)$ is a weak equivalence when $A$ is $K$-regular. For example, if $R$ is an $\ell$-algebra which is a regular supercoherent ring, then $L(E)\otimes R$ is $K$-regular (by the argument of \cite{abc}*{page 23}), so  we may replace $KH$ by $K$ to obtain the following triangle in the homotopy category of spectra which generalizes \cite{abc}*{Theorem 7.6}
\[
\xymatrix{K(R)^{(\reg(E))}\ar[r]^(.55){I-A_E^t}& K(R)^{(E^0)}\ar[r]& K(L(E)\otimes R).}
\]
In particular this applies when $R=\ell$ is a field.
 When $E^0$ is finite and $\ell$ is arbitrary, Theorem \ref{thm:basictrix} also applies to the universal homology theory $j:\aha\to kk$ of Theorem \ref{thm:ftg}. In particular, if $\#E^0<\infty$ we have a triangle in $kk$
\begin{equation}\label{tri:basic2}
\xymatrix{j(\ell^{\reg(E)})\ar[r]^{I-A_E^t}&j(\ell^{E^0})\ar[r]& j(L(E)).}
\end{equation}
In particular $L(E)$ belongs to the bootstrap category of \cite{kkwt}*{Section 8.3} whenever $E^0$ is finite, or equivalently, when $L(E)$ is unital \cite{libro}*{Lemma 1.2.12}. 
\end{ex}

\begin{rem}\label{rem:01} When $E$ is finite, we can also fit $L(E)$ into a $kk$-triangle associated to a matrix with entries in $\{ 0,1 \}$. Let $B'_E \in \{ 0,1 \}^{(E^1\coprod\sink(E)) \times (E^1\coprod\sink(E))}$, 
\[
(B'_E)_{x,y} =\left\{\begin{matrix} \delta_{r(x),s(y)}& x,y\in E^1\\ \delta_{r(x),y} & x\in E^1,y\in\sink(E)\\ 0 & x\in\sink(E)\end{matrix}\right.
\]
The matrix $B'_E=A'_{E'}$ is the incidence matrix of the maximal out-split graph $E'$ of \cite{libro}*{Definition 6.3.23}. Since by \cite{libro}*{Proposition 6.3.25},
$L(E)\cong L(E')$ in $\aha$, \eqref{tri:basic2} gives a triangle
\[
\xymatrix{j(\ell^{E^1})\ar[r]^(0.4){I-B^t_E}& j(\ell^{E^1\coprod\sink(E)})\ar[r]& j(L(E)).}
\]
Here $I,B^t_E\in (E^1\coprod\sink(E)) \times E^1$ are obtained from the identity matrix and from $(B'_E)^t$ by removing the columns corresponding to sinks.
\end{rem}

\begin{rem}\label{rem:extkk'} In \cite{kkwt}, a functor $j':\aha\to kk'$ was constructed that is universal for those homotopy invariant and $M_\infty$-stable homology theories which are excisive with respect to all, not just the linearly split short exact sequences of algebras \eqref{extensionE}. The suspension functor in $kk'$ is induced by Wagoner's suspension \eqref{def:wagsus}; we have $\Omega^{-1}j=j\Sigma$. The universal property of $j$ implies that there is a triangulated functor $F:kk\to kk'$ such that $j'=Fj$, and it follows from \cite{kkwt}*{Theorem 8.2.1} that $F:KH_n(R)=kk_n(\ell,R)\to kk_n'(\ell,R)$ is an isomorphism for all $n\in \Z$ and $R\in\aha$. Note that when $E^0$ is finite and $E^1$ is countable, Theorem \ref{thm:basictrix} applies to $X=j'$. It follows that $F_n:kk(L(E),R)\to kk_n'(L(E),R)$ is an isomorphism for all $n\in\Z$
and $R\in\aha$. In particular, if $R$ is unital, $E^1$ is countable and $E^0$ is finite, then for the $\cExt$-group we have a natural map
\[
\cExt(L(E),R)\to kk_{-1}(L(E),R).
\]
\end{rem}
\begin{conve} From now on, every statement about the image under $j$ of the Cohn or Leavitt path algebras of finitely many graphs $E_1,\dots,E_n$ will refer to the $\sqcup_{i=1}^nE_i$-stable, homotopy invariant, excisive homology theory $j:\aha\to kk$.
\end{conve}

\begin{prop}\label{prop:kkiso} Let $E$ and $F$ be graphs and $\theta\in kk( L(E),L(F))$. Assume that $E^0$ and $F^0$ are finite and that $KH_i(\theta)$ is an isomorphism for $i=0,1$. Then $\theta$ is an isomorphism. In particular $KH_n (\theta)$ is an isomorphism for all $n \in \Z$.
\end{prop}

\begin{proof}
The map $\theta$ induces a natural transformation $\theta_A:kk(A,L(E))\to kk(A,L(F))$ $(A\in\aha)$. Our hypothesis that $KH_i(\theta)$ is an isomorphism for 
$i=0,1$ says that $\theta_{\Omega^{-i}j(\ell)}$ is an isomorphism. Since $F^0$ is finite by assumption, this implies that also $\theta_{\Omega^{-i}j(\ell^{F^0})}$ and $\theta_{\Omega^{-i}j(\ell^{\reg(F)})}$ are isomorphisms. Hence applying $\theta:kk(-,L(E))\to kk(-,L(F))$ to the triangle
\[
\xymatrix{j(\ell^{\reg(F)})\ar[r]^(.55){I-A_F^t}&j(\ell^{F^0})\ar[r]& j(L(F))}
\]
and using the five lemma, we obtain that $\theta_{L(F)}$ is an isomorphism. In particular there is an element $\mu\in kk(L(F),L(E))$ such that $\mu\theta=1_{L(F)}$. Our hypothesis implies that $KH_i(\mu)$ must be an isomorphism for $i=0,1$. Hence reversing the role of $E$ and $F$ in the previous argument shows that 
$\mu$ has a left inverse. It follows that $\theta$ is an isomorphism.\end{proof}

\begin{rem}\label{rem:ojokkiso} The conclusion of Proposition \ref{prop:kkiso} does not follow if we only assume that there are group isomorphisms $\theta_i:KH_i(L(E))\iso KH_i(L(F))$ $(i=0,1)$. For example, over $\ell=\Q$, $K_0(L_0)=K_0(L_1)=\Z$ and 
$K_1(L_0)=\Q^*\cong \Z/2\Z\oplus\Z^{(\N)}\cong K_1(L_1)$. However $L_0$ and $L_1$ are not isomorphic in $kk$, since they have different periodic cyclic homology: $HP_1(L_0)=0$ and $HP_1(L_1)=\Q$. 
\end{rem}
\section{A structure theorem for Leavitt path algebras in \emph{kk}}\label{sec:leinkk}

\begin{stan}\label{stan} From here on, we shall assume that the commutative base ring $\ell$ satisfies the following conditions.
\begin{itemize}
\item[i)] $KH_{-1}(\ell)=0$.
\item[ii)] The natural map $\Z=K_0(\Z)=KH_0(\Z)\to KH_0(\ell)$ is an isomorphism.
\end{itemize}
 Moreover, all graphs considered henceforth are assumed to have finitely many vertices. In particular, all Leavitt path algebras will be unital.  
\end{stan}

\begin{rem}\label{rem:stansupercoh}
Any regular supercoherent ground ring $\ell$ satisfies standing assumption i), and moreover any Leavitt path algebra over $\ell$ is $K$-regular. Hence all statements of this section are valid for regular supercoherent $\ell$ satisfying standing assumption ii), with $K_0$ substituted for $KH_0$. In particular, this applies when $\ell=\Z$ or any field.   
\end{rem}

\begin{defi}\label{defi:ext}
Let $L(E)$ the Leavitt path algebra associated to the graph $E$. Put
\[
KH^1(L(E))=kk_{-1}(L(E),\ell).
\]
It follows from \eqref{tri:basic2} and the standing assumptions that, abusing notation, and writing $I$ for $I^t$, 
\begin{equation}\label{prevextell}
KH^1(L(E))\cong \coker( I-A_E : \Z^{E^0} \to \Z^{\reg(E)}).
\end{equation}
\end{defi}

\begin{prop}\label{prop:extkk}(Compare \cite{ck}*{Theorem 5.3}.) 
Let $E$ be a graph with finitely many vertices, such that $E^1$ is countable and $\sour(E)=\emptyset$. Then the natural map of Remark \ref{rem:extkk'} is a surjection
\begin{equation}\label{map:ontoext}
\cExt(L(E))\onto K H^1(L(E)).
\end{equation}
\end{prop}
\begin{proof} 
Our hypothesis on $E$ imply that, with the notation of\eqref{parribajo}, we have $\#\cP_v=\#\N$ for all $v\in E^0$. Hence by \eqref{map:isok}, 
$\cK(E)\cong M_\infty\ell^{\reg(E)}$, and \eqref{ses:cele} is an extension of $L(E)$ by $M_\infty\ell^{\reg(E)}$. Let $\psi:L(E)\to \Sigma(\ell)^{\reg(E)}$ be its classifying map and for $v\in\reg(E)$ let $\pi_v:\Sigma(\ell)^{\reg(E)}\to\Sigma(\ell)$ be the projection, and put $\psi_v=\pi_v\psi$. With the notation of Remark \ref{rem:extkk'} we have a triangle in $kk'$
\[
j(\ell^{E^0})\to j(L(E))\overset{\psi}\lra j(\Sigma(\ell)^{\reg(E)})\to j(\Sigma(\ell)^{E^0}).
\]
Applying $kk'(-,\Sigma(\ell))$ to it and using Remark \ref{rem:extkk'} we see that $KH^1(L(E))$ is generated by the $kk$-classes of the $\psi_v$; since these are in the image of \eqref{map:ontoext}, it follows that the latter map is surjective.    
\end{proof}

\begin{lem}\label{lem:extk0}
\item[i)] The groups $KH^1(L(E))$ and $KH_0(L(E))$ have isomorphic torsion subgroups. 
\item[ii)] $\#\sing(E)=\rk(KH_0(L(E))-\rk(KH^1(L(E))$. 
\end{lem}
\begin{proof} 
Let $D=\diag(d_1,\dots,d_n,0,\dots,0)\in \Z^{E^0\times \reg(E)}$, $d_i\ge 2$, $d_i\backslash d_{i+1}$ be the Smith normal form of $I-A^t_E$. Then $D^t$ is the Smith normal form of $I-A_E$, whence
\begin{equation}\label{invafac}
\tors KH_0(L(E))=\bigoplus_{i=1}^n\Z/d_i=\tors KH^1(L(E)).
\end{equation}
Similarly, 
\begin{align*}
\rk KH_0(L(E))-\rk KH^1(L(E))=&(\#E^0-\rk(I-A_E))-(\#\reg(E)-\rk(I-A_E)))\\
=&\#\sing(E).
\end{align*}
\end{proof}
We shall write 
\[
\tau(E)=\tors KH_0(L(E)).
\]
For $ 0\le n\le\infty$, let $\cR_n$ be the graph with exactly one vertex and $n$ loops and let $L_n=L(\cR_n)$. Thus $L_0=\ell$, $L_1=\ell[t,t^{-1}]$ is the algebra of Laurent polynomials and for $2\le n<\infty$, $L_n=L(1,n)$ is the Leavitt algebra of \cite{leav}. By \eqref{tri:basic2}, $j(L_\infty)\cong j(L_0)$ and we have a distinguished triangle in $kk$
\begin{equation}\label{tri:Ln}
j(\ell)\stackrel{n-1}{\lra}j(\ell)\lra j(L_n)\qquad (n\ge 1).
 \end{equation}

\begin{thm}\label{thm:structure}
Let $E$ be a graph with finitely many vertices. Assume that $\ell$ satisfies the standing assumptions \ref{stan}. Let $d_1,\dots,d_n$ , $d_i\backslash d_{i+1}$ be the invariant factors of the finite abelian group $\tau(E)$, $s=\#\sing(E)$ and  $r=\rk(KH^1(L(E))$. Let $j:\aha\to kk$ be the universal excisive, homotopy invariant, $E$-stable homology theory. Then 
\[
j(L(E))\cong j(L_0^s\oplus L_1^r \oplus \bigoplus_{i=1}^nL_{d_i+1}).
\]
\end{thm}
\begin{proof}
Let $D=\diag(d_1,\dots,d_n,0,\dots,0)\in \Z^{E^0\times \reg(E)}$. Then there are $P \in \Gl_{\#E^0}\Z$, $Q\in \Gl_{\# \reg(E)}\Z$ such that $P (I-A_E^t) Q = D$ where $D :=  diag(d_1, \dots, d_r, 0, \dots, 0)$. Hence we have the following commutative square in $kk$ with vertical isomorphisms
\[
\xymatrix{j(\ell^{\reg(E)})\ar[d]^{Q^{-1}}\ar[r]^{I-A_E^t}& j(\ell^{E^0})\ar[d]^P\\
j(\ell^{\reg(E)})\ar[r]_D&j(\ell^{E^0})}
\]
Hence both rows have isomorphic cones. By \eqref{tri:basic2}, the cone of the top row is $L(E)$; by \eqref{tri:Ln} and Lemma \ref{lem:extk0} that of the bottom  row is $L_0^s\oplus L_1^r \oplus \bigoplus_{i=1}^nL_{d_i+1}$. 
	\end{proof}

\begin{coro}\label{coro:kkk0ext}
The following are equivalent for graphs $E$ and $F$ with finitely many vertices.
\item[i)] $j(L(E)) \cong j(L(F))$.
\item[ii)] $KH_0(L(E)) \cong KH_0(L(F))$ and $KH^1(L(E)) \cong KH^1(L(F))$.
\item[iii)]$KH_0(L(E)) \cong KH_0(L(F))$ and $\#\sing(E) = \#\sing(F)$.
\end{coro}
\begin{proof}
Immediate from Lemma \ref{lem:extk0} and Theorem \ref{thm:structure}.
\end{proof}

\begin{rem}\label{rem:rt}
Let $E$ and $F$ be as in Corollary \ref{coro:kkk0ext}. Assume in addition that $\ell$ is a field, that $L(E)$ and $L(F)$ are simple and that $\inf(E)\ne\emptyset\ne \inf(F)$. In \cite{rt}*{Theorem 7.4}, E. Ruiz and M. Tomforde show that under these assumptions condition iii) of Corollary \ref{coro:kkk0ext} is equivalent to the existence of a Morita equivalence between $L(E)$ and $L(F)$. It follows that for such $E$ and $F$, the algebras $L(E)$ and $L(F)$ are isomorphic in $kk$ if and only if they are Morita equivalent. Ruiz and Tomforde show also that under the additional assumption that the group of invertible elements $U(\ell)$ has no free quotients, the condition that $\#\sing(E) = \#\sing(F)$ in iii) can be replaced by the condition that $K_1(L(E))\cong K_1(L(F))$. The additional assumption guarantees that $\rk(K_1(L(E)))=\rk(\ker(1-A_E^t))=\rk(KH^1(L(E))$ whenever $\#E^0<\infty$, so that $\#\sing(E)=\rk(K_0(L(E))-\rk(K_1(L(E)))$. 
\end{rem}

\section{A canonical filtration in \emph{kk(L(E),R)}}\label{sec:canofil}

Let $\ell$ be a ground ring satisfying the Standing assumptions \ref{stan}, let $E$ be a graph with finitely many vertices, $L(E)$ its Leavitt path algebra over $\ell$, and $n\in\Z$. It follows from \eqref{tri:basic1} that we have an exact sequence
\begin{equation}\label{seq:cup}
0\to KH_n(\ell)\otimes KH_0(L(E))\lra KH_n(L(E))\to\ker((I-A_E^t)\otimes KH_{n-1}(\ell))\to 0.
\end{equation} 
\begin{lem}\label{lem:otimesK0phi}
The map $KH_n(\ell)\otimes KH_0(L(E))\lra KH_n(L(E))$ of \eqref{seq:cup} is the cup product map of Example \ref{ex:tenso}.
\end{lem}
\begin{proof}
Because by assumption \ref{stan} (ii), $KH_0(\ell)=\Z$, for any finite set $X$, the cup product of Example \ref{ex:tenso} gives an isomorphism
\begin{equation}\label{map:otimes}
\cup:KH_n(\ell)\otimes KH_0(\ell^{X})\iso KH_n(\ell^{X}).
\end{equation}
Hence by \eqref{tri:basic1} we have a commutative diagram with exact rows
\[
\xymatrix{KH_n(\ell^{\reg(E)})\ar[r]^{I-A_E^t}&KH_n(\ell^{E^0})\ar[r]& KH_n(L(E))\\
          KH_n(\ell)\otimes KH_0(\ell^{\reg(E)})\ar[u]^\cup\ar[r]_{I-A_E^t}&KH_n(\ell)\otimes KH_0(\ell^{E^0})\ar[u]^\cup\ar[r]& KH_n(\ell)\otimes KH_0(L(E))\ar[u]^\cup.}
\]
\end{proof} 
Let $R$ be an algebra and $n\in\Z$. Consider the map 
\begin{equation}\label{map:khi}
KH_n:kk(L(E),R)\to \Hom_\Z(KH_n(L(E)),KH_n(R)).
\end{equation}
Define a descending filtration $\{kk(L(E),R)^i\mid 0\le i\le 2\}$ on $kk(L(E),R)$ as follows. Let  
\begin{gather}\label{filtra}
kk(L(E),R)^0=kk(L(E),R),\ \ kk(L(E),R)^1=\ker KH_0,\\
  kk(L(E),R)^2=(\ker KH_1)\cap kk(L(E),R)^1.
\end{gather}
It follows from the definition of $kk(L(E),R)^0$ and $kk(L(E),R)^1$ that $KH_0$ induces a canonical homomorphism
\begin{equation}\label{map:slice0}
kk(L(E),R)^0/kk(L(E),R)^1\to \hom(KH_0(L(E)),KH_0(R)).
\end{equation}
Let $\xi\in kk(L(E),R)^1$; by Lemma \ref{lem:otimesK0phi}, $KH_1(\xi)$ vanishes on the image of $KH_1(\ell)^{(E^0)}$, whence it induces a map $\ker(I-A_E^t)\to KH_1(R)$. Thus we have a map
\begin{equation}\label{map:slice1}
kk(L(E),R)^1/kk(L(E),R)^2\to \hom(\ker(I-A_E^t),KH_1(R)).
\end{equation} 
Let $\xi\in kk(L(E),R)^2$; embed $\xi$ into a distinguished triangle
\begin{equation}\label{tri:cxi}
C_\xi\to L(E)\overset{\xi}\lra R.
\end{equation}
We have an extension of abelian groups
\begin{equation}\label{seq:extxi}
(\cE(\xi))\quad 0\to KH_1(R)\to K_0(C_\xi)\to KH_0(L(E))\to 0.
\end{equation}
Let
\begin{equation}\label{map:slice2}
kk(L(E),R)^2\to \Ext_\Z^1(KH_0(L(E)),KH_1(R)),\quad \xi\mapsto [\cE(\xi)].
\end{equation}

\begin{thm}\label{thm:slice}
Let $E$ be a graph with finitely many vertices, $\ell$ a ring satisfying the Standing assumptions \ref{stan}, $L(E)$ the Leavitt path algebra over $\ell$ and $R$ an $\ell$-algebra. Then the maps \eqref{map:slice0}, \eqref{map:slice1} and \eqref{map:slice2} are isomorphisms. 
\end{thm}
\begin{proof}
Observe that if $B$ is an algebra and $X$ a finite set, then the isomorphism \eqref{map:kkkh} induces an isomorphism $kk_n(\ell^X,B)\iso \hom(\Z^{X},KH_n(B))$. Using this and applying $kk(-,R)$ to the triangle \eqref{tri:basic2} we obtain an exact sequence
\begin{multline}\label{seq:fil}
\Hom(\Z^{E^0},KH_1(R))\to \Hom(\Z^{\reg(E)},KH_1(R)))\to kk(L(E),R)\\
\to\Hom(\Z^{E^0},KH_0(R))\to \Hom(\Z^{\reg(E)},KH_0(R)).
\end{multline}
Since 
\begin{equation}\label{seq:freereso}
0\to \ker(I-A_E^t)\to \Z^{\reg(E)}\to\Z^{E^0}\to KH_0(L(E))\to 0
\end{equation}
is a free $\Z$-module resolution, the kernel of the last map in \eqref{seq:fil} is\goodbreak 
\noindent $\Hom(KH_0(L(E)),KH_0(R))$, and it is straightforward to check that the induced surjection
\[
kk(L(E),R)\onto\Hom(KH_0(L(E)),KH_0(R))
\]
is precisely the map $KH_0$ of \eqref{map:khi}. Hence the cokernel of the first map in \eqref{seq:fil} is $kk(L(E),R)^1$, and again because \eqref{seq:freereso} is a free resolution, we have a short exact sequence
\begin{multline}\label{seq:fil2}
0\to\Ext^1_\Z(KH_0(L(E),KH_1(R))\to kk(L(E),R)^1\to\\
 \Hom(\ker(I-A_E^t),KH_1(R))\to 0.
\end{multline}
It is again straightforward to check that the surjective map from $kk(L(E),R)^1$ in \eqref{seq:fil2} is \eqref{map:slice1}. Hence
by \eqref{seq:fil2} we have an isomorphism 
\begin{equation}\label{map:isoext}
kk(L(E),R)^2\iso \Ext^1_\Z(KH_0(L(E),KH_1(R)) 
\end{equation}
It remains to show that the above isomorphism agrees with \eqref{map:slice2}.

 Let $\xi\in kk(L(E),R)^2$ and let $\partial:j(L(E))\to \Omega^{-1}j(\ell)^{\reg(E)}$ be the boundary map in \eqref{tri:basic2}. Because $KH_0(\xi)=0$, there is an element $\hat{\xi}\in kk_1(\ell^{\reg(E)},R)$ such that $\xi=\hat{\xi}\partial$. Hence because $kk$ is triangulated, there exists $\theta\in kk(\ell^{E^0},C_\xi)$ such that we have a map of distinguished triangles
\[
\xymatrix{
j(\ell)^{\reg(E)}\ar[d]^{\Omega j\hat{\xi}}\ar[r]&j(\ell)^{E^0}\ar[d]^\theta\ar[r]&j(L(E))\ar@{=}[d]\ar[r]^{\partial}&\Omega^{-1}j(\ell)^{\reg(E)}\ar[d]^{\hat{\xi}}\\
\Omega j(R)\ar[r]&C_\xi\ar[r]&j(L(E))\ar[r]_\xi& j(R)}
\] 
Applying the functor $kk(\ell,-)$ and using that $KH_1(\xi)=0$, we obtain a map of extensions
\begin{equation}\label{map:seqfils}
\xymatrix{
\ker(I-A_E^t)\ar[d]\ar[r]&\Z^{\reg(E)}\ar[d]^{\hat{\xi}}\ar[r]&\Z^{E^0}\ar[d]^\theta\ar[r]& K_0(L(E))\ar@{=}[d]\ar[r]&0\\
0\ar[r]&K_1(R)\ar[r]&K_0(C_\xi)\ar[r]&K_0(L(E))\ar[r]& 0}
\end{equation}
By definition, \eqref{map:isoext} maps $\xi$ to the class $[\hat{\xi}]$ of $\hat{\xi}$ modulo the image of 
\goodbreak $\Hom(\Z^{E^0},KH_1(R))$. It is clear from \eqref{map:seqfils} that $[\hat{\xi}]=[C_\xi]$.
\end{proof}

\begin{coro}\label{coro:xi=0}
Let $\xi\in kk(L(E),R)$ and let $C_\xi$ be as in \eqref{tri:cxi}. Then $\xi=0$ if and only if $KH_0(\xi)=KH_1(\xi)=0$ and the extension \eqref{seq:extxi} is split.
\end{coro}

In the next corollary we shall use the fact that, since $\ker(I-A_E^t)$ is a free abelian group, the canonical surjection $KH_1(L(E))\to\ker(I-A_E^t)$ admits a section

\begin{equation}\label{map:sect}
\gamma:\ker(I-A_E^t)\to KH_1(L(E)).
\end{equation}
The map $\gamma$ induces a natural transformation 
\[
\gamma^*:\Hom(KH_1(L(E)),-)\to \Hom(\ker(I-A_E^t),-)).
\]
 
\begin{coro}\label{coro:uct} (UCT) For every $n\in \Z$ we have an exact sequence
\begin{multline*}
0\to \ext^1_\Z(KH_0(L(E)), KH_{n+1}(R))\to kk_n(L(E), R)\overset{[KH_0,\gamma^*KH_1]}\lra\\
\Hom(KH_0(L(E)),KH_n(R))\oplus \Hom (\ker(I-A_E^t),KH_{n+1}(R))\to 0.
\end{multline*}
\end{coro}
\begin{proof} In view of \eqref{loop} we may assume that $n=0$. 
By Theorem \ref{thm:slice} the map $KH_0:kk(L(E),R)\to \hom(KH_0(L(E)),KH_0(R))$ is a surjection; by definition, its kernel is $kk(L(E),R)^1$, and $\gamma^*KH_1$
induces the map \eqref{map:slice1}, which is surjective by Theorem \ref{thm:slice}. Hence $[KH_0,\gamma^*KH_1]$ is surjective, and its kernel is by definition $kk(L(E),R)^2$, which, again by Theorem \ref{thm:slice}, is $\ext^1_\Z(KH_0(L(E)),KH_1(R))$. 
\end{proof}

\begin{lem}\label{lem:kk1kh}
Let $E$ be a graph and $R$ an algebra. Assume that $\#E^0<\infty$. Then the composition map induces an isomorphism
\[
KH^1(L(E))\otimes KH_1(R)\iso kk(L(E),R)^1
\]
\end{lem}
\begin{proof} By our Standing assumptions, $KH_{-1}\ell=0$; it follows from this that\goodbreak $KH^1(L(E))=kk_{-1}(L(E),\ell)^1$ and that the composition map  lands in $kk(L(E),R)^1$. In particular,  writing $^\vee$ for the dual group, we have $KH^1(L(E))/kk_{-1}(L(E),\ell)^2=\ker(I-A_E^t)^\vee$; since the latter is free, tensoring with $KH_1(R)$ we obtain the top exact sequence of the commutative diagram below; the bottow row is exact by Theorem \ref{thm:slice}.
\[
\xymatrix{\Ext_\Z^1(\tau(E),\Z)\otimes KH_1(R)\ar[d]\ar@{ >-}[r]&KH^1(LE)\otimes KH_1(R)\ar[d]\ar@{->>}[r]& \ker(I-A_E^t)^\vee\otimes KH_1(R)\ar[d]\\
\Ext_\Z^1(\tau(E),KH_1(R))\ar@{ >-}[r]& kk(L(E),R)^1\ar@{->>}[r]& \Hom_\Z(\ker(I-A_E^t),KH_1(R)).}
\]
One checks, using the fact that for a free, finitely generated group $L$, $L^\vee\otimes(-)\cong\Hom_\Z(L,-)$, that the vertical arrows on the right and left are isomorphisms; it follows that the vertical arrow at the middle is an isomorphism as well. 
\end{proof}

\begin{lem}\label{lem:homk0} Let $E$ and $R$ be as in Lemma \ref{lem:kk1kh}. There is an exact sequence
\[
\ker(I-A_E)\otimes KH_0(R)\hookrightarrow\Hom(KH_0(L(E)),KH_0(R))\onto\tor^1_\Z(KH^1(L(E)),KH_0(R)).
\]
\end{lem}
\begin{proof} It follows from \eqref{tri:basic2} that we have a free $\Z$-module resolution
\[
0\to \ker(I-A_E)\to (\Z^{E^0})^\vee\to (\Z^{\reg(E)})^\vee\to KH^1(L(E))\to 0.
\]
Now tensor by $KH_0(R)$ and observe that $$\ker((I-A_E)\otimes \id_{KH_0(R)})=\hom(KH_0(L(E)),KH_0(R)).$$
\end{proof}

\begin{prop}\label{prop:kun} (K\"unneth theorem) 
Let $L(E)$ and $R$ be as in Theorem \ref{thm:slice} and $n\in\Z$. Then there is an exact sequence
\begin{multline*}
0\to KH^1(L(E))\otimes KH_{n+1}(R)\oplus \ker(I-A_E)\otimes KH_n(R)\to kk(L(E),R)\\
\to\tor^1_\Z(KH^1(L(E)),KH_n(R))\to 0.
\end{multline*}
\end{prop}
\begin{proof}
It suffices to prove the proposition for $n=0$. By Theorem \ref{thm:slice} we have a canonical surjection $\pi: kk(L(E),R)\to \Hom(KH_0(L(E)),KH_0(R))$. By Lemma \ref{lem:homk0} we have an inclusion
\begin{equation}\label{map:inckt}
\inc: \ker(I-A_E)\otimes KH_0(R)\subset\Hom(KH_0(L(E)),KH_0(R)).
\end{equation}
Let  $Q=\pi^{-1}(\ker(I-A_E)\otimes KH_0(R))$; by Lemmas \ref{lem:kk1kh} and \ref{lem:homk0} we have exact sequences
\begin{gather}
0\to Q\to kk(L(E),R)\to \tor^1_\Z(KH^1(L(E)),KH_0(R))\to 0\nonumber\\
0\to KH^1(L(E))\otimes KH_{1}(R)\to Q\to \ker(I-A_E)\otimes KH_0(R)\to 0.\label{seq:Q}
\end{gather}
We have to show that the second sequence above splits. Let $\theta:\ker(I-A_E)\to KH^0(L(E))$ be a section of the canonical projection. One checks that for $\inc$ as in \eqref{map:inckt}, the composite 
\[
\theta':\ker(I-A_E)\otimes KH_0(R)\overset{\theta\otimes \id}\lra KH^0(L(E))\otimes KH_0(R) \overset{\circ}\lra kk(L(E),R)
\]
satisfies $\pi\theta'=\inc$. It follows that the sequence \eqref{seq:Q} splits, completing the proof. 
\end{proof}

\begin{rem}\label{rem:general}
The key property of the algebra $B=L(E)$ that we have used in this section is that for some 
$m,n\in\N$ and $M\in\Z^{m\times n}$ we have a distinguished triangle in $kk$
\[
j(\ell)^n\overset{M}\lra j(\ell)^m\to j(B).
\]
All the results and proofs in this section apply to any algebra $B$ with the above property, substituting $M$ for $I-A_E^t$, and assuming of course that $\ell$ satisfies the Standing assumptions \ref{stan}. However one can show, using the Smith normal form of $M$, that any such $B$ is $kk$-isomorphic to the sum of Leavitt path algebra and a number of copies of the suspension $\Omega_{-1}\ell$. 
  
\end{rem} 
\end{document}